\documentclass[11pt,a4paper]{amsart}

\usepackage{amssymb}
\usepackage{amsmath}
\usepackage{amsthm}
\usepackage{latexsym}

\newtheorem{teo}{Theorem}[section]
\newtheorem{defi}[teo]{Definition}
\newtheorem{lem}[teo]{Lemma}
\newtheorem{pro}[teo]{Proposition}
\newtheorem{cor}[teo]{Corollary}
\newtheorem{rem}[teo]{Remark}

\newcommand{\idg}{\mathfrak{e}}
\newcommand{\G}{\mathbb{G}}

\newcommand{\g}{\mathfrak{g}}
\newcommand{\R}{\mathbb{R}}
\newcommand{\mean}{-\!\!\!\!\!\!\int}
\newcommand{\fa}{\ \ \ \ \forall \ }
\newcommand{\n}[1]{\left|#1\right|_{_{\mathbb{G}}}}
\newcommand{\prs}[3]{\langle #1 , #2 \rangle_{_#3}}

\newcommand{\h}[1]{{#1}'}
\newcommand{\nh}[1]{{#1}''}
\newcommand{\moll}[1]{\rho_n(\n{#1})}

\begin{document}

\title{Approximations of Sobolev norms in Carnot groups}

\author{Davide Barbieri}

\thanks{The author has been supported by University of Bologna, Universit\'e de Cergy-Pontoise, Universit\'e Franco-Italienne Programme Vinci, and by European Project GALA.}

\maketitle

\begin{abstract}
This paper deals with a notion of Sobolev space $W^{1,p}$ introduced by J.Bourgain, H.Brezis and P.Mironescu by means of a seminorm involving local averages of finite differences. This seminorm was subsequently used by A.Ponce to obtain a Poincar\'e-type inequality. The main results that we present are a generalization of these two works to a non-Euclidean setting, namely that of Carnot groups. We show that the seminorm expressd in terms of the intrinsic distance is equivalent to the $L^p$ norm of the intrinsic gradient, and provide a Poincar\'e-type inequality on Carnot groups by means of a constructive approach which relies on one-dimensional estimates. Self-improving properties are also studied for some cases of interest.
\end{abstract}


\section{Introduction}
This paper is based on a sequence of works by J.Bourgain, H.Brezis and
P.Mironescu \cite{BBM}, H.Brezis \cite{B} and A.Ponce \cite{Po} concerning a
peculiar approach to Sobolev spaces $W^{1,p}$, which allows to obtain Sobolev norms
in terms of integral averages of finite differences, suitably localized by means of
radial mollifiers.

This technique provides finite scale approximations of first derivatives which can
be used to define Sobolev $(1,p)$ norms that are equivalent to the
ordinary ones and reduce to them by letting this scale become infinitely small. 

We developed a generalization of the main results concerning this technique to the
setting of Carnot groups, which represent a class of non-Euclidean metric spaces widely studied in the
context of degenerate PDE's, complex analysis and control theory, lying at the basis of
the so-called sub-Riemannian geometry.

A key result of the cited approximating technique is expressed by \cite[Corollary 1]{BBM}, where it is shown
that for $f \in W^{1,p}(\Omega)$, $\Omega$ being a smooth bounded
domain in $\R^N$, it holds
$$
\int_{\Omega}\frac{|f(y)-f(x)|^p}{|y-x|^p}\rho_n(|y-x|)dy
\stackrel{n \to \infty}{\longrightarrow} K|\nabla
f(x)|^p \qquad \textrm{in}\ \ L^1(\Omega)
$$
where $K=K(p,N)$ and $\rho_n$ are radial mollifiers: nonnegative
functions such that
\begin{displaymath}
\int \rho_n(|x|)dx = 1\ ,\ \ \textrm{and}\ \ \ \lim_{n \to
\infty}\int_\delta^\infty \rho_n(r) r^{N-1} dr = 0 \fa \delta
> 0\ .
\end{displaymath}
Various applications of this method can be found in the
cited references, and among them there is the availability of a
Poincar\'e-type inequality, see \cite{Po}, which makes use of the
described integral approximation: for $f \in L^p(\Omega)$
$$
\int_\Omega |f(x)-f_\Omega|^pdx\ \leq\ C\ \int_\Omega
\int_\Omega\frac{|f(y)-f(x)|^p}{|y-x|^p}\rho_n(|y-x|)dydx\fa n > n_0
$$
$f_\Omega$ being the average of $f$ on $\Omega$, $C$ a constant
which depends on $p,N$ and on the constant of the ordinary
Poincar\'e inequality and $n_0=n_0(C,\Omega)$.

In order to extend to Carnot groups this finite differences approach to Sobolev norms,
originally exploited in a purely Euclidean setting, we could apply techniques similar
to those used in \cite{B} by making use of the differentiable structure introduced in \cite{Pa}
(see also \cite{BLU}, \cite{F}, \cite{FL}) and of the homogeneity of the
space. In particular our proof shows how the essential properties needed to obtain results as sharp as 
in the Euclidean case can be obtained in the case of Carnot groups.

For what concerns the Poincar\'e inequality, our proof is essentially different
from the one given in \cite{Po}: due to the non-Euclidean composition law of the group, the original procedure
cannot be applied, so we provided a new constructive proof that relies on a one-dimensional inequality.
The main advantage of this approach consists of the fact that it shows quantitatively the relationship between the size of the bounded domain on which the mean oscillation of the function is evaluated and the scale at which the gradient is approximated by finite differences.

Carnot groups \cite{BLU}, \cite{S}, \cite{VSC}, \cite{FS}, \cite{G}, \cite{FNafsa} are Lie
groups endowed with a homogeneous structure given by a family of
anisotropic dilations. The Lie algebras of Carnot groups are
stratified: a linear subspace sufficient to generate the whole
algebra can be defined from the commutation relations, or
equivalently from the group composition law. This subspace is
generally called the horizontal layer and, thanks to the Chow
connectivity theorem, it defines an intrinsic metric structure,
called the Carnot-Carath\'eodory distance, which is given by the infimum length
among all the integral curves of horizontal vector fields connecting
two points. Other equivalent metric structures on Carnot groups are
those induced by a so-called homogeneous norm by making use of group
operations. The homogeneous structure induces another peculiarity of
Carnot groups: the Hausdorff dimension is in general higher than the
topological dimension. This phenomenon is due to the different
scalings of the directions, which contribute to the measure in a non
uniform way; from the algebra point of view, it reflects the fact
that some directions are ``heavier'' than others, because they
belong to higher layers.

In what follows, $\G$ will indicate a Lie group, and $\g$
a Lie algebra. By Lie algebra $\g$ of a Lie group $\G$, we will mean
the usual linear space of left-invariant vector fields on $\G$, with
the Lie bracket given by the commutator: $\big([X,Y]f\big)(x)\
\dot{=}\ \big(X(Yf)\big)(x) - \big(Y(Xf)\big)(x), \ \ X,Y \in \g$.
The identification of $\g$ with $T_\idg\G$, the tangent space at the
identity $\idg$ of $\G$, will also be assumed. As a global reference
on this topic, we suggest \cite{V}. Here we will only recall briefly
the notion of exponential map for $\G$ a Lie group and $\g$
its Lie algebra: let $\gamma_{X,x}(t)$ the integral curve of $X \in
\g$ passing through $x \in \G$
\begin{eqnarray*}
\gamma_{X,x}'(t) & = & X(\gamma_{X,x}(t))\\
\gamma_{X,x}(0) & = & x\ .
\end{eqnarray*}
The exponential map is then defined as
\begin{eqnarray*}
\exp: & \g & \longrightarrow \G\\
& X & \longrightarrow \exp(X)\ \dot{=}\ \gamma_{X,\idg}(1)
\end{eqnarray*}
and it is a local diffeomorphism of a neighborhood of\ $0 \in \g$
to a neighborhood of\ $\idg \in \G$:
$$
d\exp(0) = \mathbb{I}_\g\ ,
$$
which extends to a global diffeomorphism if $\G$ is simply
connected, so it allows to construct systems of coordinates on $\G$
given a basis $\big\{X_j\big\}_{j=1}^N$ of \ $\g$. In particular an
element $x \in \G$ is said to have canonical coordinates
 $(x_1,\dots,x_N)\in\R^N$\ if\ \ $x=\exp\
(x_1X_1+\dots+x_NX_N)$.

\vspace{6pt}\noindent {\bf Acknowledgments}\\
This paper is a part of author's PhD thesis, written under the supervision of Thierry Coulhon and Bruno Franchi, to whom the author is particularly indebted for their careful directions. The author thanks also Augusto Ponce and Francesco Serra Cassano for helpful discussions.


\section{Carnot Groups}

\begin{defi}[Carnot groups]
A Carnot group $\G$ is a simply connected Lie group with a
stratified Lie algebra $\g$ \emph{i.e.} a nilpotent Lie algebra that
can be decomposed into a direct sum of linear subspaces called
layers, $\{W_j\}_{j=1}^k$, in such a way that the first layer
generates the whole algebra:\vspace{-4pt}
$$
\g = \bigoplus_{j=1}^k\ W_j \ , \ \ W_{j+1} = [W_1,W_{j}] \ \
\forall \ j<k \ , \ \ W_{j}=0 \ \ \forall \ j>k\ .
$$
The positive integer $k$ is called the step of $\g$.
\end{defi}
We now state some basic properties of such groups.\vspace{6pt}\\
\textsc{Dilations:} given a stratified Lie algebra\
$\g$, it is possible to define algebra dilations as a 1-parameter
group $\{\Delta_\lambda\}_{\lambda> 0}$ of automorphisms of\ $\g$
which act on layers as
$$
\Delta_\lambda (X) = \lambda^j X \fa X \in W_j\ .
$$
Algebra dilations induce on $\G$, through exponential mapping, a
1-parameter group of automorphisms of $\G$ called group dilations
$\{\delta_\lambda\}_{\lambda> 0}$:
$$
\delta_\lambda\left(\delta_\mu(x)\right) = \delta_{\lambda\mu}(x)\ ,
\ \delta_\lambda(x\cdot y) = \delta_\lambda(x)\cdot
\delta_\lambda(y) \fa x,y\ \in\ \G \fa \lambda,\mu >0\ .
$$
The quantity $Q = \sum j\ \textrm{dim}(W_j)$ is called the
homogeneous dimension of $\G$, which equals its
Hausdorff dimension as we will see in Corollary
\ref{Changevar}.\vspace{6pt}\\
\textsc{Identification with $\R^N$:} since the exponential map is a
global diffeomorphism, we can identify \ $\G$ with\ $\R^N$ endowed
with the induced group law (which we shall denote by $\cdot$), $N$
being the linear dimension of $\g$, by making use of canonical
coordinates. If we set $m_j =$ \textrm{dim}$(W_j)$, then $\sum m_j =
N$, which differs from $Q$ if $\g$ has more than one layer.\\
Given a basis $\{X_1, \dots, X_{m_1},X_{m_1+1}, \dots, X_{m_1+m_2},
\dots, X_N\}$ of $\g$, grouped with respect to increasing layers, we can
correspondingly order the canonical coordinates so that the action
of group dilations reads
$$
\delta_\lambda(x_1,\dots,x_N) = (\lambda\ x_1, \dots, \lambda\
x_{m_1}, \lambda^2\ x_{m_1+1}, \dots, \lambda^2\ x_{m_1+m_2}, \dots,
\lambda^k\ x_N)\ .
$$

Accordingly, we can group the canonical coordinates into distinct
families $\{x_1,\dots,x_{m_1}\}, \dots,
\{x_{m_1+\dots+m_{k-1}},\dots,x_{N}\}$, and the ones corresponding
to the $i$-th layer will be said to have weight
$i$.

Because of the special role played by the first layer of $\g$
(that defines the geometry of $\G$), it turns
out to be convenient to distinguish the variables of weight 1 by
putting
$$
x=(\h{x},\nh{x}) \ , \ \textrm{with} \ \h{x} =
(x_1,\dots,x_{m_1}) \ \ \textrm{and} \ \ \nh{x} =
(x_{m_1+1},\dots,x_{N})
$$
Vector operations on $\R^N$ will be performed as
operations on $\G$ and in particular the identity $\idg$ of\ \ $\G$
is given by the vector $0 \in \R^N$.
\vspace{6pt}\\
\textsc{Sub-Riemannian structure:} the first layer, or horizontal
layer, of the algebra $\g$ plays a central role in the geometry of
Carnot groups. It can be identified with a linear subspace of the
tangent space of $\G$ at the origin and defines, by left
translations, a canonical subbundle $H\G$ of the tangent bundle
$T\G$, whose fibers at points $x \in \G$ will be denoted by $H_x\G$.
In particular $\G$ can be endowed with a sub-Riemannian metric: a
scalar product $\prs{.}{.}{X}$ defined on each horizontal fiber
$H_x\G$ that makes the vector fields $X_1, \dots, X_{m_1}$
orthonormal. The corresponding norm will be indicated as
$|.|_{_X}$.\\We say that a continuously differentiable curve
$\gamma:[0,T] \rightarrow \G$ is horizontal if
$\dot{\gamma}(t) \in H_{\gamma(t)}\G$ for all $0\leq t\leq T$.
Chow connectivity theorem asserts that any two points of $\G$
can be connected by a horizontal curve, settling the well-posedness of the
following definition.
\begin{defi}[Carnot-Carath\'eodory distance]
The CC distance is the intrinsic distance on Carnot groups, defined
as the infimum sub-Riemannian length of horizontal curves connecting
two points (see \cite{FL}, \cite{NSW}, \cite{G}):
\begin{equation}\label{CCdist}
d_X(x,y) \doteq \begin{array}[t]{c} \textrm{inf}\\
\scriptstyle{\gamma \ \textrm{\emph{horizontal}}} \end{array}
\int_0^1 |\dot{\gamma}(t)|_{_X}dt
\end{equation}
and we denote by $B_X(x,r)$ the $d_X$ metric balls.
\end{defi}
Due to the left invariance of the vector fields, $d_X$ is left
invariant, that is
$$
d_X(x,y) = d_X(p\cdot x, p\cdot y) \fa x,y,p \in \G\ .
$$
Moreover, since it involves only horizontal vector fields, it is homogeneous
of degree one with respect to group dilations:
$$
d_X(\delta_\lambda(x),\delta_\lambda(y)) = \lambda d_X(x,y) \fa x,y
\in \G \ , \ \lambda > 0 \ .
$$

The metric space $(\G,d_X)$ is geodesic: for any
couple of points there exists a minimizer in the class of horizontal
curves. This is due to Chow Theorem and to the fact that $\G$ is
locally compact and complete (see e.g. \cite[\S 9]{Hein}).

\begin{defi}[Horizontal gradient] For a smooth function $f:\G \rightarrow \R$ we define its horizontal
gradient as
\begin{equation}\label{horgrad}
\nabla_X f = \sum_{j=1}^{m_1} (X_jf) X_j
\end{equation}
The map $\nabla_X f$ defines a horizontal vector field. The
coordinates of $\nabla_X f(x)$ in $H_x\G$ with respect to the chosen
orthonormal basis $\{X_1(x),\dots,X_{m_1}(x)\}$ are given by
$((X_1f)(x),
\dots, (X_{m_1}f)(x))$.
\end{defi}

Whilst both the sub-Riemannian structure and the horizontal gradient
are basis dependent, it is still possible to associate with a
function $f$, for which $\{X_jf(x)\}_{j=1}^{m_1}$ exist for all $x
\in \G$, a well defined linear functional $\mathcal{L}^f$ on $H\G$ that
does not depend on the basis of the first layer:
\begin{eqnarray}\label{horgradinv}
\mathcal{L}^f_x : & H_x\G & \longrightarrow \R\nonumber\\
& V & \longmapsto \prs{\nabla_Xf(x)}{V}{X}\ .
\end{eqnarray}
\begin{defi}[Homogeneous norms]
A homogeneous norm is a continuous function $ \n{.} : \G
\longrightarrow \R^+ $ with the following properties:
\begin{enumerate}
\item $\n{x} = 0 \Leftrightarrow x = 0$\ ;
\item $\n{x^{-1}} = \n{x}$\ ;
\item $\n{\delta_\lambda(x)} = \lambda \n{x}$\ .
\end{enumerate}
\end{defi}
A homogeneous norm induces a left-invariant homogeneous distance by
\begin{equation}\label{homdist}
d(x,y) = \n{y^{-1}\cdot x}
\end{equation}
and a quasi triangular inequality holds (see \cite[Proposition 1.6]{FS}):
\begin{equation}\label{quasitr}
d(x,y) \leq \alpha \left( d(x,z) + d(z,y)\right) \ \ \ \textrm{for
some}\ \alpha > 0\ .
\end{equation}

With $B(x,r)$ we will indicate the metric ball with respect to the
homogeneous distance (\ref{homdist}):
$B(x,r) = \{y \in \G\ \textrm{such that}\ d(x,y) < r\}$.

A particular homogeneous norm can be constructed starting from the
Carnot-Carath\'eodory distance, by $|x|_{CC} = d_X(x,0)$, and we
recall that any homogeneous norm is equivalent to $|.|_{CC}$, as for
the respective distances: if we call $\lambda$ the equivalence parameter, then
\begin{equation}\label{equiv}
\lambda^{-1} d_X(x,y) \leq d(x,y) \leq \lambda d_X(x,y)\ .
\end{equation}

\begin{defi} [Homogeneous homomorphisms]
Let $\left(\G_1,\cdot\right)$ and $\left(\G_2,\centerdot\right)$ denote
Carnot groups with dilations $\delta_\lambda^{(i)}$, $i=1,2$.
A map $\varphi:\G_1 \rightarrow \G_2$ is called a homogeneous
homomorphism if
\begin{itemize}
\item[i)] $\varphi(\delta_\lambda^{(1)}(x)) =
\delta_\lambda^{(2)}(\varphi(x)) \fa x \in \G_1\ \ , \lambda > 0$\ ;
\item[ii)] $\varphi(x\cdot y) = \varphi(x)\centerdot\varphi(y) \fa x,y \in \G_1$\ .
\end{itemize}
\end{defi}

\begin{rem} [$\R^M$-valued homogeneous homomorphisms]\label{RemHom}
If $(\G_2,\centerdot)=(\R^M,+)$, dilations $\delta_\lambda^{(2)}$
are Euclidean. Homogeneous homomorphisms $\varphi : \left(\G,\cdot
\right)\rightarrow \left(\R^M,+ \right)$ reduce then (see
\cite[Proposition 2.5]{FSSC_S2}) to functions from the weight 1
variables of \ $\G$ to $\R^M$, with the property
$$
\varphi\big(\delta_{\lambda_1}(h_1)\cdot\delta_{\lambda_2}(h_2)\big)
= \lambda_1\varphi(h_1) + \lambda_2\varphi(h_2)\fa h_1, h_2 \in \G,
\lambda_1, \lambda_2 \in \R\ .
$$
To emphasize the fact that $\varphi$ acts only on the horizontal
components of points $x \in \G$ we will write $\varphi(\h{x})$.
\end{rem}

\begin{defi}[$\G$-linear functions]
A homogeneous homomorphism $\varphi$ with values in $\R^M$ will be called a $\G$-linear
function, after \cite{FNafsa}. Its action is actually a matrix product in
horizontal coordinates, defined by $Mm_1$ real constants
$\{A^\varphi_{ij}\},\ i=1\dots M\ ,\ j=1\dots m_1$\ and reads
explicitly
$$\varphi(\h{x}) = \left(\sum_{j=1}^{m_1} A^\varphi_{1j} x_j\ , \dots \ , \
\sum_{j=1}^{m_1} A^\varphi_{Mj} x_j \right)\ .
$$
The real-valued matrix $(A^\varphi)_{ij}$ will be called its representative matrix.

In the case of $M=m_1$ and $A^\varphi \in GL(m_1,\R)$, we obtain a
 class of maps from $\G$ to $\G$ which we call
$\G$-change of basis, by a $\phi^A:\G \rightarrow \G$ such that
\begin{eqnarray}\label{G-change}
\phi^A: & x & \longmapsto \ \ \ \ \xi\nonumber\\
& (\h{x}\ ,\nh{x}) & \longmapsto (\varphi(\h{x}),\nh{x})\ .
\end{eqnarray}
Such a transformation indeed represents the change of canonical
coordinates in $\G$ resulting from a change of basis in the
horizontal layer of the algebra.\\
By horizontal rotation we will mean a $\G$-change of basis with $A^\varphi \in O(m_1,\R)$.
\end{defi}

\begin{rem}
There always exists a homogeneous norm which is invariant under horizontal rotations
(see \cite[Ch.13\, \S 7.12]{S}). One concrete realization of such a norm is obtained in \cite{FSSC_S2}, Theorem
5.1.
\end{rem}

We now introduce the commonly used notion of
differentiability on Carnot groups, given by Pansu in \cite{Pa}, for
the case of real-valued functions.
\begin{defi} [Pansu differentiability of real valued functions]\label{defiPdiff}
We say that a function $f:\G\rightarrow\R$ is P-differentiable at
$x\in\G$ if there exists a $\G$-linear map $L^f_x:\G\rightarrow\R$
such that
\begin{equation}\label{Plim}
\lim_{\n{h}\to 0} \frac{f(x\cdot h) - f(x) - L^f_x(\h{h})}{\n{h}} =
0 \ .
\end{equation}
\end{defi}
Let $H(x) = \sum_{i=1}^{m_1} h_i X_i(x) \in H_x\G$ and $h = \exp(H)
= (\h{h},0) \in \G$. It is well known that if $f$ is
P-differentiable at $x$, then
\begin{displaymath}
L^f_x(\h{h}) = \frac{d}{dt} f(x\cdot \exp(tH))\bigg|_{t=0} = (Hf)(x)
= \mathcal{L}^f_x(H(x))
\end{displaymath}
where $\mathcal{L}^f_x$ is given by (\ref{horgradinv}), so that the Pansu differential does not depend on the basis chosen
for $\g$. In the rest of the paper we will use the notation
$$
L^f_x(\h{h}) = \prs{\nabla_Xf}{\h{h}}{X}\ .
$$

Integration on Lie groups is performed with respect to the
Haar measure: the following proposition provides its construction on
Carnot groups.
\begin{pro}\label{Int}
Let $dX$ be the Lebesgue measure on the linear space $\g$. Then the
Haar measure $d\mu(x)$ on $\G$ is given by the image through
exponential mapping of $dX$. This means that, given
$f:\G\rightarrow\R$, the integration on $\G$ can be expressed as an
integral in canonical coordinates, that is, on $\R^N$:
\begin{equation}\label{haar}
\int_\G f(x)d\mu(x) \ \ \dot{=} \
\int_{\R^N}f(x_1,\dots,x_N)dx_1\dots dx_N\ \footnote{For the sake of simplicity
we keep the same notation for $f$ on the group and on $\R^N$.}.
\end{equation}
\end{pro}

The main point here is that integration on $\G$ can be
performed as in $\R^N$, keeping in mind that group operations affect
the measure through canonical coordinates changes: for this reason
we will indicate the measure $d\mu(x)$ simply as $dx$. The following
corollary states three basic properties of the measure under changes
of variables.
\begin{cor}\label{Changevar}
Let $\G$ be a Carnot group and $f:\G \longrightarrow \R$ an
integrable function on $\G$, then the Haar measure of $\G$
\begin{itemize}
\item[1.] is invariant under left and right translations (unimodularity):
$$
\int_\G f(\alpha\cdot x) dx = \int_\G f(x\cdot\alpha) dx = \int_\G
f(x) dx\fa \alpha \in \G\ ;
$$
\item[2.] scales under group dilations by the homogeneous dimension of $\G$:
$$
\int_\G f(\delta_\lambda(x))dx = \lambda^Q\int_\G f(x)dx \fa \lambda
> 0 \ ;
$$
\item[3.] is affected by $\G$-changes of basis (\ref{G-change}) as the Lebesgue measure on $(\R^N,+)$, \emph{i.e.} putting $\xi\, =\,\phi^A(x)\, =\, (A\h{x},\nh{x})$ we have:
$$
\int_\G f(x) dx = \int_\G f(\xi) |\emph{det}A| d\xi\fa A\in GL(m_1,\R)\ .
$$
\end{itemize}
\end{cor}

The first two properties provide the basic relation
$$
\left|B(x,r)\right| = r^Q \left|B(0,1)\right| \fa x \in \G
$$
where $|B|$ stands for the measure of the ball $B$, so we will set $c_B = \left|B(0,1)\right|$.

A notion of smoothness for functions on Carnot groups is that
of being $\mathcal{C}^1$ with respect to horizontal vector fields.
This is stronger than being P-differentiable, but requires less
regularity than being $\mathcal{C}^1$ in Euclidean sense.
\begin{defi}
A continuous function $f:\G\rightarrow\R$ is said to be in
$\mathcal{C}^1_\G(\G,\R)$ if $X_jf:\G\rightarrow\R$ exist and are
continuous for $j=1,\dots,m_1$.
\end{defi}
\begin{teo}\label{teoC1-Pdiff}
If $f$ is a $\mathcal{C}^1_\G(\G,\R)$ function, then $f$ is
P-differentiable.
\end{teo}
\begin{rem}\label{remIncl}
$\mathcal{C}^1 \subset \mathcal{C}^1_\G$, and the inclusion is
strict.
\end{rem}
See \cite[\S 5, Theorem 7]{FSSC_H} for a proof of Theorem
\ref{teoC1-Pdiff}, and \cite[\S 5, Remark 6]{FSSC_H} for an example of the inclusion in
Remark \ref{remIncl}.

Along with these basic group-related notions of regularity, we recall the commonly
used notion of Sobolev space $W^{1,p}$.
\begin{defi}[Sobolev space $W^{1,p}_\G(\G)$]
Let $\G$ be a Carnot group and let $p \geq 1$ be fixed. We denote by
$W^{1,p}_\G(\G)$ the set of functions $f \in L^p(\G)$ such that $X_jf
\in L^p(\G)$ for $j=1,\dots,m_1$, endowed with the norm
\begin{equation}\label{sobnorm}
\|f\|_{1,p} = \left(\int_\G |f(x)|^pdx\right)^{1/p} + \left(\int_\G
|\nabla_Xf(x)|_{_X}^pdx\right)^{1/p}\ .
\end{equation}
In particular we recall that the classical density result of smooth
functions still holds, see \cite{FSSC_MS}.
\end{defi}

To conclude the section, we recall a result from \cite[Proposition 1.5]{F} (see also \cite[\S 2 Proposition 3]{KS},
 and \cite[Proposition 1.13]{FS}) which will be
useful in the computations.
\begin{lem}\label{LemFolland}
Let $\G$ be a Carnot group and let $\n{.}$ be a homogeneous norm.
Let $f:\G \rightarrow  \R$ be a homogeneous function of degree $-Q$, that is
$f(\delta_\lambda(x)) = \lambda^{-Q}f(x)\ \forall x \in \G$, and
locally integrable away from $0$. Then there exists a constant
$M(f)$ such that
\begin{equation}\label{FollandInt}
\int_\G f(x) g(\n{x}) d\mu(x) = M(f) \int_0^\infty g(r)\
\frac{dr}{r}
\end{equation}
for any $g:\R^+\rightarrow\R$ \ such that each side makes sense.
\end{lem}

We compute explicitly the constant $M$ for the 
function $\displaystyle{f(x) = \n{x}^{-Q}}$:
\begin{displaymath}
c_B = \int_{B(0,1)} d\mu(x) = \int_\G \frac{1}{\n{x}^Q}
\n{x}^Q\chi_{[0,1]}(\n{x}) d\mu(x) = M\left(\frac{1}{\n{x}^Q}\right)
\int_0^1 r^{Q-1} dr
\end{displaymath}
hence $M\left(\n{x}^{-Q}\right) = Q c_B$.
This allows the explicit computation of integrals on balls of
functions depending only on the distance from the center of the ball
in terms of integrals on the real line with no need of any coarea
formula:
$$
\int_{B(x_0,R)} f(\n{x_0^{-1}\cdot x}) d\mu(x) = \int_{B(0,R)}
f(\n{y}) d\mu(y) = Q c_B \int_0^R f(r) r^{Q-1} dr\ .
$$


\section{The Bourgain-Brezis-Mironescu Sobolev space\label{sec:BS}}

We begin providing a notion of radial mollifiers on Carnot groups.
\begin{defi} [Radial mollifiers]\label{Radmol}
Let $\G$ be a Carnot group and let $\n{.}$ be a given homogeneous
norm. A family of functions $\{\rho_n\}_{n> 0}$, $\rho_n : \G
\longrightarrow \R^+$, is said to be a family of radial mollifiers
with respect to the norm $\n{.}$ if each $\rho_n(x)$ depends only on $\n{x}$
and
\begin{eqnarray*}
i) & & \int_\G \rho_n(x) dx = 1\fa n > 0\ ;\\
ii) & & \int_{\G \smallsetminus B(0,\delta)} \rho_n(x) dx
\stackrel{n \rightarrow \infty}{\longrightarrow} \ 0 \fa \delta > 0\
.
\end{eqnarray*}
\end{defi}
\begin{lem}
Let $\left\{\tilde{\rho}_n\right\}$ be a family of functions
$\tilde{\rho}_n : \R^+ \longrightarrow \R^+$, and set $\rho_n: \G
\longrightarrow \R^+$ as $\rho_n(x) = \tilde{\rho}_n(\n{x})$. Then
$\rho_n$ are radial mollifiers in the sense of Definition
\ref{Radmol} if and only if it holds
\begin{itemize}
\item[1.] $ \displaystyle{\int_0^\infty \tilde{\rho}_n(r)r^{Q-1} dr =
\frac{1}{Qc_B}\fa n > 0}\ ; $
\item[2.] $\displaystyle{\int_\delta^\infty \tilde{\rho}_n(r)r^{Q-1} dr
\stackrel{n \rightarrow \infty}{\longrightarrow} \ 0 \fa \delta >
0}\ .$
\end{itemize}
\end{lem}
\begin{proof}
As for $1.$, by direct computation, using Lemma \ref{LemFolland}, we
obtain
\begin{displaymath}
\int_\G \rho_n(x) dx = \int_\G \frac{1}{\n{x}^Q}
\tilde{\rho}_n(\n{x}) \n{x}^Q dx = Q c_B \int_0^\infty
\tilde{\rho}_n(r) r^{Q-1} dr\ .
\end{displaymath}
Assertion $2.$ can be handled in the same way.
\end{proof}

As customary when dealing with radial functions, with
abuse of notation we will denote a radial mollifier $\rho_n$ as
$\rho_n(\n{x})$ without distinguish $\rho$ and $Qc_B\tilde{\rho}$.
Moreover, we will make also use of their
1-dimensional counterpart: mollifiers $\rho_n^{(1)} : \R^+
\longrightarrow \R^+$ satisfying
\begin{equation}\label{Onedimmoll}
\int_0^\infty \rho_n^{(1)}(\tau) d\tau = 1 \qquad \int_\delta^\infty
\rho_n^{(1)}(\tau) d\tau \stackrel{n \rightarrow
\infty}{\longrightarrow} \ 0 \fa \delta > 0
\end{equation}
which, due to the preceding corollary, are in 1-1 correspondence
with the radial mollifiers of Definition \ref{Radmol} by
\begin{equation}\label{1toQ}
\rho_n^{(1)}(r) = Q c_B \rho_n(r) r^{Q-1}\ .
\end{equation}

Before stating the main result of this section, we prove
two basic lemmata which will be crucial in the proof. The first
lemma is a uniformity result concerning the convergence of the limit
(\ref{Plim}) in the case of functions with continuous derivatives,
and the second lemma is a direct consequence of Lemma
\ref{LemFolland} which provides the exact constant for Theorem
\ref{main}.

Theorem \ref{main}, that is the counterpart in Carnot groups of
\cite[Theorem 2]{BBM}, however proved following \cite{B}, gives a
way to characterize Sobolev spaces and estimate their norms
(\ref{sobnorm}) by means of integral approximations: the local
structure, expressed by an integral norm of the differential, can be
replaced by a finite scale evaluation of function's oscillations,
which is shown to be an equivalent condition. This finite scale is
governed by the radial mollifiers, and in the limit of its
becoming tiny, the differential formulation is recovered.

We stress the fact that the proof of the sharp result of convergence
strongly relies on the assumption of rotational invariance for the
homogeneous norm (see the proof of Lemma \ref{LemFollandLike}).
\begin{lem}\label{LemUnif}
If $f \in \mathcal{C}_\G^1(\G,\R)$ then
\begin{equation}\label{omega}
\omega_x(h) \doteq |f(x\cdot h) - f(x) - \prs{\nabla_X
f(x)}{\h{h}}{X}|
\end{equation}
is such that
\begin{eqnarray*}
i) & & \displaystyle{ \frac{\omega_x(h)}{\n{h}} \rightarrow
0}\ \ \ \; \ \ \textrm{as}\ \ h\to 0,\ \textrm{uniformly for x on compact sets}\ ;\\
ii) & & \displaystyle{ \frac{\omega_x(h)}{\n{h}} \leq C_K } \fa x
\in K\ \textrm{compact},\ \forall\ \n{h} <1\ .
\end{eqnarray*}
\end{lem}
\begin{proof}
By Theorem \ref{teoC1-Pdiff}, $f$ is P-differentiable. Now if
$\varphi(t) = f(x \cdot \delta_t(h))$, then by the Mean Value Theorem there exists a $t^* \in (0,1)$ such that
$$
f(x \cdot h) - f(x) = \varphi(1) - \varphi(0) = \frac{d}{dt}
\varphi(t)\bigg|_{t=t^*} = \prs{\nabla_Xf(x \cdot
\delta_{t^*}(h))}{\h{h}}{X}
$$
where the last transition holds due to the horizontality of the
Pansu differential. We then get
\begin{eqnarray*}
\omega_x(h) & = & |\prs{\left(\nabla_Xf(x \cdot \delta_{t^*}(h)) -
\nabla_Xf(x)\right)}{\h{h}}{X}|\\
& \leq & |\h{h}|_{\R^{m_1}} |\nabla_Xf(x \cdot \delta_{t^*}(h)) -
\nabla_Xf(x)|_{_X}
\end{eqnarray*}
which gives
$$
\frac{\omega_x(h)}{\n{h}} \leq c |\nabla_Xf(x \cdot \delta_{t^*}(h))
- \nabla_Xf(x)|_{_X}
$$
where $c$ is such that $|\h{h}|_{\R^{m_1}} < c \n{h}$, see e.g. \cite[Lemma 1.3]{F}.

The lemma is then proved by the hypothesis of $f$ being a $\mathcal{C}_\G^1$ function.
\end{proof}

\begin{lem}\label{LemFollandLike}
Let $\G$ be a Carnot group, $\{\rho_n\}_{n > 0}$ a family of radial
mollifiers as in Definition \ref{Radmol} and $\n{.}$ a homogeneous
norm invariant under horizontal rotations. Call
\begin{equation}\label{FollandLike}
\kappa_n = \int_{B(0,1)}
\frac{|\langle\hat{v},\h{x}\rangle|^p}{\n{x}^p} \moll{x}dx
\end{equation}
with $\hat{v}$ a unit vector of $\R^{m_1}$ and $\langle. ,  .
\rangle$
the Euclidean scalar product on $\R^{m_1}$.\\
Then
\begin{enumerate}
\item[i)] for all $n>0$, $\kappa_n$ does not depend on
$\hat{v}$. In particular
$$
\kappa_n = (p+Q) \int_{B(0,1)} |\langle\hat{e_1},\h{x}\rangle|^pdx
\int_0^1 \rho_n(r) r^{Q-1} dr
$$
where $\hat{e_1}$ stands for the first unit vector of the standard
basis of $\R^m$;
\item[ii)] $\kappa_n \to \kappa \doteq \displaystyle{\frac{(p+Q)}{Q c_B} \int_{B(0,1)}
|\langle\hat{e_1},\h{x}\rangle|^pdx}$ \ as $n \to \infty$ \ .
\end{enumerate}
\end{lem}
\begin{proof}
Point $ii)$ is straightforward from $i)$ and the definition of
radial mollifiers. To prove $i)$ we make use of Lemma
\ref{LemFolland}:
\begin{eqnarray*}
\kappa_n & = & \int_\G
\left(\frac{|\langle\hat{v},\h{x}\rangle|^p}{\n{x}^{p+Q}}\right)
\moll{x}\n{x}^Q \chi_{[0,1]}(\n{x})dx\\
& = &
M\left(\frac{|\langle\hat{v},\h{x}\rangle|^p}{\n{x}^{p+Q}}\right)\int_0^1
\rho_n(r)r^{Q-1}dr \ .
\end{eqnarray*}
On the other hand,
\begin{eqnarray*}
\int_{B(0,1)} |\langle\hat{v},\h{x}\rangle|^p dx \!\!& = &\!\! \int_\G
\left(\frac{|\langle\hat{v},\h{x}\rangle|^p}{\n{x}^{p+Q}}\right)
\n{x}^{p+Q} \chi_{[0,1]}(\n{x})dx\\
& = &\!\!
M\left(\frac{|\langle\hat{v},\h{x}\rangle|^p}{\n{x}^{p+Q}}\right)
\int_0^1 r^{p+q-1}dr =
\frac{1}{p+Q}M\left(\frac{|\langle\hat{v},\h{x}\rangle|^p}{\n{x}^{p+Q}}\right)
\end{eqnarray*}
so that
$$
M\left(\frac{|\langle\hat{v},\h{x}\rangle|^p}{\n{x}^{p+Q}}\right) =
(p+Q)\int_{B(0,1)} |\langle\hat{v},\h{x}\rangle|^pdx\ .
$$
This expression does not depend on $\hat{v}$, since by an orthogonal
$\G$-change of basis (\ref{G-change}) which does not alter the
measure (Corollary \ref{Changevar}) nor the homogeneous norm (so
that also the domain of integration does not change), it is possible
to choose any other unitary vector of $\R^{m_1}$ by rotation: set $A
\in O(m_1,\R)$ and call $\hat{w} = A^T\hat{v}$, then
\begin{eqnarray*}
\int_{B(0,1)} |\langle\hat{v},\h{x}\rangle|^p dx & = & \int_{B(0,1)}
|\langle\hat{v},\h{\phi^A(x)}\rangle|^p dx =
\int_{B(0,1)} |\langle\hat{v},A\h{x}\rangle|^p dx\\
& = & \int_{B(0,1)} |\langle\hat{w},\h{x}\rangle|^p dx\ .
\end{eqnarray*}
\end{proof}

\begin{teo}\label{main}
Let $\G$ be a Carnot group, $\n{.}$ a homogeneous norm
invariant under horizontal rotations and $\{\rho_n\}_{n > 0}$ a family of radial
mollifiers as in Definition \ref{Radmol}. Let then $f \in L^p(\G)$
be given, with $1<p<\infty$. If in addition
\begin{equation}\label{condition}
I_n = \int_\G\int_\G \frac{|f(y) - f(x)|^p}{\n{x^{-1}\cdot y}^p}
\moll{x^{-1}\cdot y} dx dy \leq C \fa n > n_0
\end{equation}
then
\begin{eqnarray*}
i) & & f \in W^{1,p}_\G(\G)\ ;\\
ii) & & I_n \longrightarrow \displaystyle{\kappa \int_\G |\nabla_X f
(x)|^p_{_X} dx} \ \ \textrm{as} \ \, n \to \infty \ ;
\end{eqnarray*}
where $\kappa$ is the constant given by Lemma \ref{LemFollandLike}:
$$
\kappa = \frac{(p+Q)}{Q c_B} \int_{B(0,1)}
|\langle\hat{e_1},\h{x}\rangle|^pdx\ .
$$
\end{teo}
\begin{proof}
If we set $h = x^{-1}\cdot y$ we can write
$$
I_n = \int_\G\int_\G \frac{|f(x\cdot h) - f(x)|^p}{\n{h}^p} \moll{h}\,
dh\, dx\ .
$$
The proof can now be carried out in two steps, producing
estimates from above and below. The first step consists in
showing that for any $f \in \mathcal{C}_\G^1(\G)$
\begin{equation}\label{below}
\kappa \int_\G |\nabla_Xf(x)|_{_X}^pdx \leq \liminf_{n \to \infty}
\int_\G\int_\G\frac{|f(x\cdot h)-f(x)|^p}{\n{h}^p}\moll{h} dh dx\ .
\end{equation}
Such a result for $\mathcal{C}_\G^1(\G)$ extends then to any $f \in L^p(\G)$ by density \cite{FSSC_MS}.
Relation (\ref{below}) provides one half of $ii)$ and, together
with (\ref{condition}), it implies $i)$.\\
On the other hand, the second step of the proof will be the
complementary estimate: for any $f \in \mathcal{C}_{0_\G}^1(\G)$ it holds
\begin{equation}\label{above}
\kappa\int_\G |\nabla_X f(x)|_{_X}^p dx \geq \limsup_{n \to \infty}
\int_\G\int_\G \frac{|f(x\cdot h) - f(x)|^p}{\n{h}^p}\moll{h}dh dx\ .
\end{equation}
This together with (\ref{below}) yields, for any $f \in
\mathcal{C}_{0_\G}^1(\G)$,
$$
\lim_{n \to \infty} \int_\G\int_\G \frac{|f(x\cdot h) -
f(x)|^p}{\n{h}^p}\moll{h} dh dx = \kappa\int_\G |\nabla_X
f(x)|_{_X}^p dx
$$
which extends by density \cite{FSSC_MS} to any $f \in
W^{1,p}_\G(\G)$.

We now prove (\ref{below}). In order to do that, let us first prove
the following preliminary claim: let $K$ be a compact set, $B \equiv
B(0,1)$ and $\kappa_n$ be the same as in Lemma \ref{LemFollandLike},
then
\begin{equation}\label{preclaim}
\int_Kdx\int_Bdh
\frac{|\prs{\nabla_Xf(x)}{h}{X}|^p}{\n{h}^p}\moll{h} = \kappa_n
\int_K dx |\nabla_Xf(x)|_{_X}^p\ .
\end{equation}
Indeed, since it is sufficient to restrict to compact sets $K$ where
$\nabla_Xf(x) \neq 0$,
\begin{eqnarray*}
\int_Kdx\int_Bdh&&\!\!\!\!\!\!\!\!\!\!\!\!\!\!
\frac{|\prs{\nabla_Xf(x)}{\h{h}}{X}|^p}{\n{h}^p}\moll{h}\\
& = & \int_Kdx |\nabla_Xf(x)|_{_X}^p \int_Bdh
\frac{|\prs{\nu_f(x)}{\h{h}}{X}|^p}{\n{h}^p}\moll{h}
\end{eqnarray*}
where $\displaystyle{\nu_f(x) = \frac{\nabla_Xf(x)\phantom{x}}{|\nabla_Xf(x)|_{_X}}}$
is a unit norm horizontal vector, so that this expression
corresponds to that of Lemma \ref{LemFollandLike}. In particular
$\kappa_n$ does not depend on $x$, as a result of the rotational
invariance shown in the lemma. Claim
(\ref{preclaim}) is is then proved.

Now, using the notation of (\ref{omega}), Lemma \ref{LemUnif}, we
get
$$
|\prs{\nabla_X f(x)}{\h{h}}{X}| \leq |f(x \cdot h) - f(x)| +
\omega_x(h)
$$
then for any $p>1, \theta > 0$ there exists a $C_{p,\theta}>0$\ \
such that
$$
|\prs{\nabla_X f(x)}{\h{h}}{X}|^p \leq (1+\theta)|f(x \cdot h) -
f(x)|^p + C_{p,\theta}\omega_x(h)^p\ .
$$
Combining this relation with (\ref{preclaim}), with $\theta$
arbitrarily fixed, we get
\begin{eqnarray*}
\kappa_n \displaystyle{\int_K dx} |\nabla_Xf(x)|_{_X}^p & \leq &
\displaystyle{(1+\theta)\int_K dx\int_B dh
\ \frac{|f(x \cdot h) - f(x)|^p}{\n{h}^p}\moll{h}}\\
&& + \ C_{p,\theta} \int_K dx\int_B dh
\frac{\omega_x(h)^p}{\n{h}^p}\moll{h}\ .
\end{eqnarray*}
Thus (\ref{below}) will follow after proving that the second term on
the right expression vanishes as $n \to \infty$. To see this, we use
both the radial mollifier property of vanishing tails and the
behavior of the rest given by Lemma \ref{LemUnif}.

For any $\delta \in (0,1)$ we can split the integral into two parts:
\begin{equation}\label{neglect}
\int_K dx \int_{B} dh \frac{\omega_x(h)^p}{\n{h}^p} \moll{h} = 
\int_K \int_{B(0,\delta)} + \int_K \int_{B\smallsetminus
B(0,\delta)} = J_n^{(1)} + J_n^{(2)}\ .
\end{equation}
We can see that $J_n^{(1)}$ is arbitrarily small for any $\delta$
sufficiently small, thanks to the uniform convergence of
$\omega_x(h)$, while for $J_n^{(2)}$ we can control $\omega_x(h)$
and use the tail property of $\moll{h}$. More precisely, by Lemma \ref{LemUnif}
\begin{itemize}
\item[(1)] for all $\eta > 0$ there exists \ $\overline{\delta}(\eta,K)$
\ such that\ $\displaystyle{\omega_x(h) <
\left(\frac{\eta}{2|K|}\right)^{1/p}\n{h}}$\\
for all $\n{h}<\overline{\delta}(\eta,K)$ and all $x \in K$. We then have\vspace{3pt}\\
$J_n^{(1)}\ <\ \displaystyle{\frac{\eta}{2|K|}
|K|\int_{B(0,\delta)}\moll{h}}dh\ < \
\displaystyle{\frac{\eta}{2}}\fa \delta < \overline{\delta}(\eta,K)\
\ \forall\ n$;
\item[(2)] $\omega_x(h)\ \leq\ C_K^{1/p}\ \n{h}$ for all $x\in K$.
This yields\vspace{3pt}\\
$J_n^{(2)} \ \leq\ C_K |K| \displaystyle{\int_{B(0,1)\smallsetminus
B(0,\delta)}
\moll{h} dh} \fa \delta \in (0,1)$\vspace{3pt}\\
and the integral over the annulus can be made arbitrarily small for
$n$ large due to the tail property of
$\rho_n$ (Definition \ref{Radmol}), so that\vspace{3pt}\\
$J_n^{(2)} \ \leq\ C_K |K|\ \displaystyle{\frac{\eta}{2C_K|K|}}\ =\
\displaystyle{\frac{\eta}{2}} \fa n > \overline{n}(\delta,\eta, K)$\
.
\end{itemize}
Thus, for any fixed $\eta > 0$, there exists an $\overline{n}$
depending only on $\eta$ and on the compact set $K$ such that the
integral (\ref{neglect}) is smaller than $\eta$ provided $n >
\overline{n}$:
$$
\forall\ \eta > 0\ \ \exists \ \overline{n}=\overline{n}\
(\overline{\delta}(\eta,K),\eta,K)\ \ \ \textrm{such that}\ \ \ J_n^{(1)}
+ J_n^{(2)} < \eta \fa n > \overline{n}\ .
$$
We then end up with
$$
\kappa_n \int_K |\nabla_X f(x)|_{_X}^p dx \leq (1+\theta)
\int_K\!\!dx\int_B\!\!dh \frac{|f(x\cdot h) -
f(x)|^p}{\n{h}^p}\moll{h} + C_\theta \eta
$$
which is true for any $n$ sufficiently large. This provides
$$
\int_K |\nabla_X f(x)|_{_X}^p dx \leq (1+\theta) \liminf_{n \to
\infty} \frac{1}{\kappa_n}\int_\G\int_\G \frac{|f(x\cdot h) -
f(x)|^p}{\n{h}^p}\moll{h}dhdx
$$
for any arbitrary $\theta > 0$. Moreover, since the right term does
not depend on $K$ anymore, we can enlarge the integration domain on
the left up to the whole $\G$, so that (\ref{below}) is
proved.

We come now to the second estimate, proving
(\ref{above}) for $f \in \mathcal{C}_{0_\G}^1(\G)$. Again from
(\ref{omega}), using the triangular inequality in the other way, we
get
$$
|f(x\cdot h) - f(x)| \leq |\prs{\nabla_X f(x)}{\h{h}}{X}| +
\omega_x(h)
$$
so that for any $p>1, \theta > 0$ there exists a $C_{p,\theta}>0$\ \
such that
$$
|f(x\cdot h) - f(x)|^p \leq (1+\theta) |\prs{\nabla_X
f(x)}{\h{h}}{X}|^p + C_\theta \omega_x(h)^p\ .
$$
This means that, by integration on $\G$ and making use of claim
(\ref{preclaim})
$$
\int_\G\int_\G \frac{|f(x\cdot h) - f(x)|^p}{\n{h}^p}\moll{h}dhdx \leq
(1+\theta) \kappa_n \int_\G |\nabla_X f(x)|_{_X}^p dx + J_n
$$
where
$$
J_n = C_\theta\int_\G\int_\G \frac{\omega_x(h)^p}{\n{h}^p} \moll{h} dhdx
$$
can be handled as in (\ref{neglect}), since integration can be
actually restricted to the support of $f$. This means that $J_n$
vanishes as $n \to \infty$, thus proving (\ref{above}).
\end{proof}

Theorem \ref{main} actually provides an equivalent notion
of the Sobolev space $W^{1,p}_\G(\G)$: this can be seen by the
following proposition (see \cite[Theorem 1]{BBM}).
\begin{pro}
Let $f \in W^{1,p}_\G(\G)$, $1\leq p < \infty$, and $\rho:\G
\rightarrow \R$ a non\-ne\-ga\-tive function such that $\int_\G \rho(x) dx =
1$. Then
\begin{equation}\label{stronger}
\int_\G\int_\G dx dy \frac{|f(y) - f(x)|^p}{\n{x^{-1}\cdot
y}^p}\rho(x^{-1}\cdot y) \leq \int_\G |\nabla_Xf(x)|^p_{_X} dx
\end{equation}
\end{pro}
\begin{proof}
We start with $f \in \mathcal{C}_\G^1(\G)$. Define, for $h \in \G$,
$v(t) = f(x\cdot \delta_t(h))$ and recall that, by making use of the
Pansu differential properties,
$$
\dot{v}(t) = \prs{\nabla_Xf(x\cdot\delta_t(h))}{\h{h}}{X}\ .
$$
Then
$$
|f(x\cdot h) - f(x)| = |\int_0^1
\prs{\nabla_Xf(x\cdot\delta_t(h))}{\h{h}}{X}dt| \leq
\n{h}\int_0^1 |\nabla_Xf(x \cdot \delta_t(h))|_{_X}
$$
so that
\begin{eqnarray*}
\int_{\G} |f(x\cdot h) - f(x)|^p dx & \leq & \n{h}^p\int_\G
dx \int_0^1 |\nabla_Xf(x \cdot \delta_t(h))|^p_{_X} dt\\
& = & \n{h}^p \int_\G |\nabla_Xf(x)|^p_{_X} dx\ .
\end{eqnarray*}
If we now multiply both sides by $\rho(h)/\n{h}^p$ and integrate
over $\G$ in $dh$ we obtain estimate (\ref{stronger}) for any $f \in
\mathcal{C}^1_\G(\G)$, so by density for any $f \in W^{1,p}_\G(\G)$.
\end{proof}


\section{The Poincar\'e-Ponce inequality\label{sec:PPsec}}

This section is devoted to extend to Carnot groups a
Poincar\'{e}-like inequality, introduced in \cite{Po}, where on the
right side stands the finite scale approximation of the $L^p$ norm
of the gradient described in the last section. The proof is
original, in the sense that a different technique with respect to
\cite{Po} is used, and shows how the finite scale of mollifiers
relates to a finite scale on the domain of evaluation. To this end,
we need an additional assumption on the mollifiers, namely we assume
the $\rho_n$ to be nonincreasing functions. This technical point was
already discussed in \cite{BBM} and \cite{Po}.

We start with a preliminary one dimensional lemma.
\begin{lem}[One dimensional inequality]\label{LEMonedim}
Let $f\in L^p_{loc}(\R)$ be given, $1\leq p < \infty$, and let
$\varphi:\R^+\rightarrow\R^+$ be a nonincreasing function in
$L^1_{loc}(\R^+)$. Then, setting $\overline{f} =
\displaystyle{\int_{-\frac{1}{2}}^{\frac{1}{2}} f(t)dt}$, it holds
\begin{equation}\label{onedimpoinc01}
\int_{-\frac{1}{2}}^{\frac{1}{2}} |f(t) - \overline{f}|^pdt \leq \frac{2}{\int_0^1
d\tau \rho(\tau)} \int_{-\frac{1}{2}}^{\frac{1}{2}} dt \int_{-\frac{1}{2}}^{\frac{1}{2}} ds
\frac{|f(t) - f(s)|^p}{|t - s|^p} \varphi(|t - s|)\ .
\end{equation}
\end{lem}
\begin{proof}
To prove this lemma we will make use of \cite[Lemma 2]{BBM}. We
first claim that if we set
$$
g(\tau) = \displaystyle{\frac{1}{\tau^p}\int_{-\frac{1}{2}}^{\frac{1}{2} - \tau} dt
|f(t+\tau) - f(t)|^p}
$$
then for all $\tau \in (0,1)$ it holds $g(\tau) \leq g(\tau/2)$.
Indeed if we take $0 < \sigma < 1/2$
\begin{eqnarray*}
g\!\!\!\!\!\!\!\!&&\!\!\!\!\!\!\!(2\sigma) \ = \ \frac{1}{(2\sigma)^p} \int_{-\frac{1}{2}}^{\frac{1}{2} - 2\sigma}\!\!\!dt |f(t+2\sigma) - f(t)|^p\\
&\leq& \frac{1}{2\sigma^p} \left\{\int_{-\frac{1}{2}}^{\frac{1}{2} -
2\sigma}\!\!\!dt |f(t+2\sigma) - f(t+\sigma)|^p
+ \int_{-\frac{1}{2}}^{\frac{1}{2} - 2\sigma}\!\!\!dt |f(t+\sigma) - f(t)|^p\right\}\\
&\leq& \frac{1}{\sigma^p} \int_{-\frac{1}{2}}^{\frac{1}{2} - \sigma}\!\!\!dt
|f(t+\sigma) - f(t)|^p = g(\sigma)\ .
\end{eqnarray*}

The previous claim allows to apply Lemma 2 of \cite{BBM}:
\begin{displaymath}
\int_0^1 d\tau g(\tau) \varphi(\tau) \geq \frac{1}{2} \int_0^1 d\tau
g(\tau) \int_0^1 d\tau \varphi(\tau)\ .
\end{displaymath}
The proof of inequality (\ref{onedimpoinc01}) now follows:
\begin{eqnarray}\label{righthere}
\int_{-\frac{1}{2}}^{\frac{1}{2}} \!\!\!& \!\!\!dt \!\!\!& \!\!\!|f(t) - \overline{f}\, |^p \ \leq \ 
\int_{-\frac{1}{2}}^{\frac{1}{2}} dt \int_{-\frac{1}{2}}^{\frac{1}{2}} ds |f(t) - f(s)|^p\\
& \leq & 2\int_0^1 d\tau \frac{1}{\tau^p} \int_{-\frac{1}{2}}^{\frac{1}{2} - \tau}dt
|f(t+\tau) - f(t)|^p\nonumber\\
& = & 2\int_0^1 d\tau g(\tau) \leq \frac{4}{\int_0^1 d\tau \rho(\tau)} \int_0^1 d\tau
g(\tau)\rho(\tau)\nonumber\\
& = & \frac{2}{\int_0^1 d\tau \rho(\tau)}\int_{-\frac{1}{2}}^{\frac{1}{2}} dt
\int_{-\frac{1}{2}}^{\frac{1}{2}} ds \frac{|f(t) - f(s)|^p}{|t - s|^p} \rho(|t -
s|)\ .\nonumber
\end{eqnarray}
\end{proof}

\begin{cor}\label{cor1Poinc}
Let $\{\rho_n^{(1)}\}$ be a family of 1-dimensional mollifiers as in
(\ref{Onedimmoll}) such that each $\rho_n^{(1)}$ is a nonincreasing
function. Then for any $t_0 \in \R, \ T > 0$ and for any $C > 2$
there exists an $n_0$ such that
$$
\int_I |f(t) - \overline{f}_I|^p dt \leq C T^p \int_I dt \int_I ds
\frac{|f(t)-f(s)|^p}{|t-s|^p}\rho^{(1)}_n(|t-s|)
$$
for any $n > n_0$ and for any $f\in L^p_{loc}(\R), \ 1\leq p <
\infty$, where $I$ is the interval $I(t_0,T) = [t_0 - T/2, t_0 +
T/2]$ and we have set $\displaystyle{\overline{f}_I = \mean_I
f(t)dt}$.

In particular, $n_0=n_0(C,T)$ is determined by the condition
\begin{equation}\label{threshold}
\int_0^T \rho^{(1)}_n(\tau)d\tau
> \frac{2}{C} \fa n
> n_0
\end{equation}
and always exists by definition of mollifiers.
\end{cor}

This corollary says that Lemma \ref{LEMonedim} is a
Poincar\'e-like inequality on intervals of $\R$. The proof
essentially relies on applying scaling and translations.
\begin{proof}
If we take any $\varphi:\R^+\rightarrow\R^+$ in $L^1_{loc}(\R^+)$,
nonincreasing, and set
$$
\tilde{f}(t) = f(Tt + t_0) \ , \ \ \tilde\varphi(t) = \varphi(Tt)\ .
$$
we obtain auxiliary functions which still satisfy the hypotheses of
Lemma \ref{LEMonedim}, so that (\ref{onedimpoinc01}) holds. Hence we
obtain, for the couple $(f,\varphi)$, a Poincar\'e-like inequality
on the interval :
\begin{equation}\label{onedimpoinc}
\int_I |f(t) - \overline{f}_I|^p dt \leq \frac{2}{\int_0^T
\varphi(\tau)d\tau} T^p \int_I dt \int_I ds
\frac{|f(t)-f(s)|^p}{|t-s|^p}\varphi(|t-s|)\ .
\end{equation}
Now it suffices to write inequality (\ref{onedimpoinc})
for a family of one dimensional mollifiers $\rho^{(1)}_n$ as in
(\ref{Onedimmoll}).
\end{proof}

Condition (\ref{threshold}), that allows to determine the threshold $n_0$,
means that the smaller is the interval of evaluation, the smaller
must be the analyzing scale powered by the mollifiers. This is made
clearer by the following remark.

\begin{rem}\label{REMsigma}
Set $\varphi$ within the hypotheses of Lemma \ref{LEMonedim}, then
for any $\sigma>0$
\begin{displaymath}
\int_{-\frac{1}{2}}^{\frac{1}{2}}\!\!dt |f(t) - \overline{f}|^p \leq
\frac{2}{\int_0^{\sigma} \varphi(\tau)d\tau} \sigma
\int_{-\frac{1}{2}}^{\frac{1}{2}}\!\! dt \int_{-\frac{1}{2}}^{\frac{1}{2}}\!\! ds
\frac{|f(t)-f(s)|^p}{|t-s|^p}\varphi(\sigma|t-s|)\ .
\end{displaymath}
This follows immediately from inequality (\ref{onedimpoinc01}) with
$\tilde\varphi(\tau) = \varphi(\sigma\tau)$.
\end{rem}

If we take here $\varphi$ as a mollifier $\rho^{(1)}_n$,
we obtain a right term with $\sigma\rho^{(1)}_n(\sigma\ .\ )$, which
has the same mass of $\rho$ but a different analyzing scale (larger
as $\sigma$ tightens). For small $\sigma$, this can be compensated
by taking higher values of $n$, which indeed provide a constant that
does not depend on $\sigma$. The point we want to stress is that to
get Poincar\'e-like estimates with a nonlocal term on the right, the
scale of nonlocality (indexed by n) is necessarily influenced by the
scale of the problem.

In this section we will also need a result on the
structure of Carnot-Ca\-rath\'e\-odo\-ry balls from \cite[\S 3]{Morb},
which we state here in a shortened form for the case of interest.
\begin{teo}\label{Morbidelli}
Let $\G$ be a Carnot group of dimension $N$ and step $k$,
denote by $m_j$ the dimension of the $j-th$ layer and by $\{X_1,
\dots, X_{m_1}\}$ a basis of the horizontal layer. Let $d_X$ be the
CC distance (\ref{CCdist}) and $B_X$ the related balls. Then there
exist three multi-indexes of length $M=\sum_{l=1}^{k}(3\cdot
2^{l-1} - 2)m_l$
\begin{eqnarray*}
I & = \ (i_1,\dots,i_M) \ , \ & i_n \in \{1,\dots,m_1\}\\
J & = \ (j_1,\dots,j_M) \ , \ & j_n \in \{1,\dots,N\}\\
\omega & = \ (\omega_1,\dots,\omega_M) \ , \ & \omega_n \in \{0,1\}
\end{eqnarray*}
and two geometric constants $0<b<a<1$ such that if we set
\begin{eqnarray*}
E_{I,J,\omega}: & \R^N & \longrightarrow \ \ \G\\
& (t_1,\dots,t_N) & \longmapsto \ \
\exp\left((-1)^{\omega_1}t_{j_1}X_{i_1}\right)\cdot \ldots \cdot
\exp\left((-1)^{\omega_M}t_{j_M}X_{i_M}\right)
\end{eqnarray*}
then
\begin{equation}\label{Control}
B_X(0,bR) \subset E_{I,J,\omega}(Q(0,aR)) \subset B_X(0,R)
\end{equation}
for all $R>0$, where
$$
Q(0,\delta) = \left\{(t_1,\dots,t_N)\in \R^N : \
\emph{max}\,|t_j|<\delta \right\}\ .
$$
\end{teo}

The main geometric meaning of this theorem is contained in
the structure of the multi-indexes $I,J,\omega$, which we have
skipped here since it goes beyond our needs. The idea is that each
point of a Carnot group can be reached by exponential curves of
horizontal vector fields, and when a commutator of two vector fields
is needed, it can be approximated by a finite length ``square path''
along the two fields, taken successively with opposite signs. This
is the main difference between this result and the classical result
due to Nagel, Stein and Wainger (\cite[Theorem 7]{NSW}), where
instead the authors work directly with integral curves of
commutators. The key relation (\ref{Control}) indicates that the
approximated exponential coordinates $(t_1,\dots,t_N)$ are uniformly
controlled by the size of the CC-ball that they are able to span,
and this will be used as a control on the time-lengths of connecting
curves.

We are now ready to prove the main theorem of
this section.
\begin{teo}\label{PPteo}
Let $\G$ be a Carnot group and $\n{.}$ a homogeneous norm with
parameter $\lambda$ as in (\ref{equiv}). Then there exist two positive geometric
constants $\beta,\mu$ and a positive constant $C_{p,Q}$ depending only on $p$, $Q$ and the
geometric constants $a, b, M$ given by Theorem \ref{Morbidelli} such that
\begin{equation}\label{PP}
\int_{B}\!\! dx |f(x) - f_B|^p \leq C_{p,Q} \ R^p\int_{\mu
B}\int_{\mu B}\!\!\!\! dx dy \ \frac{|f(y) - f(x)|^p}{\n{y^{-1}\cdot x}^{p+Q-1}}
\frac{\varphi(\n{y^{-1}\cdot x})}{\int_0^{\beta R}
\varphi(\tau)d\tau}
\end{equation}
for all balls $B$ of radius $R$, all $f$ in $L^p_{loc}(\G), \ 1 \leq
p < \infty$ and all nonincreasing $\varphi : \R^+ \rightarrow \R^+$
in $L^1_{loc}(\R^+)$.
\end{teo}
\begin{proof}
The proof is largely inspired by a simple proof of the
classical Poincar\'{e} inequality, due to Varopoulos (see
\cite{Varopoulos}, and also \cite{HK}, \cite{FNafsa}). The constant $C$ that will appear is intended as a positive constant that can change from line to line. To start, by H\"older inequality
$$
I = \int_{B(x_0,R)}\!\!\!\!\!\!dx |f(x) - f_B|^p \leq \frac{1}{c_B
R^Q} \int_{B(x_0,R)}\!\!\!\!\!\!dx \int_{B(x_0,R)}\!\!\!\!\!\!dy
|f(x) - f(y)|^p\ .
$$
Then, passing to Carnot-Carath\'eodory balls and setting $c_{B_X}=|B_X(0,1)|$
\begin{eqnarray*}
I & \leq & \frac{1}{c_B R^Q} \int_{B_X(x_0,\lambda R)}\!\!\!\!\!\!dx
\int_{B_X(x_0,\lambda R)}\!\!\!\!\!\!dy |f(x) - f(y)|^p\\
& = & \lambda^Q \frac{c_{B_X}}{c_B}\mean_{B_X(x_0,\lambda
R)}\!\!\!\!\!\!dy \int_{y^{-1}\cdot
B_X(x_0,\lambda R)}\!\!\!\!\!\!dz |f(y\cdot z) - f(y)|^p\\
& \leq & \lambda^Q \frac{c_{B_X}}{c_B}\mean_{B_X(x_0,\lambda
R)}\!\!\!\!\!\!dy \int_{B_X(0,2\lambda R)}\!\!\!\!\!\!dz |f(y\cdot
z) - f(y)|^p
\end{eqnarray*}
where the last transition holds because $y^{-1}\cdot B_X(x_0,\lambda
R) \subset B_X(0,2\lambda R)$: indeed, by left-invariance of the
CC-distance
$$
d_X(z,0) = d_X(y\cdot z, y) \leq d_X(y\cdot z, x_0) + d_X(x_0,y) <
2\lambda R
$$
so we end up with
$$
I \leq C \mean_{B_X(0,2\lambda
R)}\!\!\!\!\!\!dz \int_{B_X(x_0,\lambda R)}\!\!\!\!\!\!dy |f(y\cdot
z) - f(y)|^p\ .
$$

By Theorem \ref{Morbidelli} we can reach any $z \in B_X(0,2\lambda R)$ with a finite composition
of integral curves of horizontal vector fields, within a finite time. More precisely we can write
$$
z = \prod_{k=1}^M\exp\left((-1)^{\omega_k}t_{j_k}X_{i_k}\right)
$$
where, by (\ref{Control}), if we set $T = \displaystyle{\frac{2\lambda}{b}}$ we have that the times
needed to cover this path are uniformly bounded by
\begin{equation}\label{T}
|t_n| < a T R \fa n \in \{1,\dots,N\}\ .
\end{equation}

Making use of partial compositions $\displaystyle{\zeta_n = \prod_{k=1}^n\exp\left((-1)^{\omega_k}t_{j_k}X_{i_k}\right)}$, we obtain 
\begin{eqnarray*}
|\!\!\!\!&&\!\!\!\!\!\!\!\!\!\!\!f(y\cdot z) - f(y)|^p \ = \  \left|f\big(y\cdot\prod_{k=1}^M\exp\left((-1)^{\omega_k}t_{j_k}X_{i_k}\right)\big)-f(y)\right|^p\\
& = &
\Bigg|\sum_{n=1}^{M-1}\bigg[f\big(y\cdot\zeta_{n+1}\big)
-
f\big(y\cdot\zeta_n\big)\bigg] + f(y\cdot\exp((-1)^{\omega_1}t_{j_1}X_{i_1}))-f(y)\Bigg|^p\\
& \leq &
M^{p-1}\Bigg[\sum_{n=1}^{M-1}\bigg|f\big(y\cdot\zeta_{n+1}\big)
- f\big(y\cdot\zeta_n\big)\bigg|^p\\
&& \phantom{M^{p-1}} + \bigg|f(y\cdot\exp((-1)^{\omega_1}t_{j_1}X_{i_1}))-f(y)\bigg|^p\Bigg]
\end{eqnarray*}
so that we can separate the estimate of the mean oscillations into $M$ pieces
\begin{equation}\label{POINT}
I \leq C \sum_{n=0}^{M-1} I_n
\end{equation}
where
\begin{displaymath}
\left\{
\begin{array}{rcl}
I_n & = & \displaystyle{\mean_{B_X(0,2\lambda R)}\!\!\!\!\!\!dz
\int_{B_X(x_0,\lambda R)}\!\!\!\!\!\!dy} \ \displaystyle{
\bigg|f\big(y\cdot\prod_{k=1}^{n+1}\exp\left((-1)^{\omega_k}t_{j_k}X_{i_k}\right)\big)}\\
&& - \ \displaystyle{ f\big(y\cdot\prod_{k=1}^n\exp\left((-1)^{\omega_k}t_{j_k}X_{i_k}\right)\big)\bigg|^p} \ \ \ \ \textrm{if} \ n>0\\
I_0 & = & \displaystyle{\mean_{B_X(0,2\lambda R)}\!\!\!\!\!\!dz
\int_{B_X(x_0,\lambda R)}\!\!\!\!\!\!dy} \ 
\displaystyle{\bigg|f(y\cdot\exp((-1)^{\omega_1}t_{j_1}X_{i_1}))-f(y)\bigg|^p}\
.
\end{array}
\right.
\end{displaymath}
We note here that terms $I_n$ with $n>0$ can be reduced to the form
of term $I_0$ by means of the change of variables $\eta = y\cdot\zeta_n$
\begin{eqnarray*}
I_n & = & \mean_{B_X(0,2\lambda R)}\!\!\!\!\!\!dz
\int_{B_X(x_0,\lambda R)\cdot\zeta_n}\!\!d\eta
\left|f(\eta\cdot\exp((-1)^{\omega_{n+1}}t_{j_{n+1}}X_{i_{n+1}}))-f(\eta)\right|^p\\
& \leq & \mean_{B_X(0,2\lambda R)}\!\!\!\!\!\!dz
\int_{B_X\left(x_0,\left(T+\lambda\right)R\right)}\!\!\!\!\!\!d\eta
\left|f(\eta\cdot\exp((-1)^{\omega_{n+1}}t_{j_{n+1}}X_{i_{n+1}}))-f(\eta)\right|^p
\end{eqnarray*}
where the last transition holds since, by (\ref{T}) and making use
of (\ref{Control})
\begin{displaymath}
d_X(\eta,x_0) \leq d_X(\eta,y) + d_X(y,x_0) < d_X(\zeta_n,0) +
\lambda R < T R + \lambda R\ .
\end{displaymath}
Each $I_n$, for $n=0,\dots,M-1$ , is then bounded by an integral of
type
$$
J = \mean_{B_X(0,2\lambda R)}\!\!\!\!\!\!dz
\int_{B_X(x_0,(T+\lambda) R)}\!\!\!\!\!\!d\eta
\left|f(\eta\cdot\exp(\tau X))-f(\eta)\right|^p
$$
where $|\tau| < aTR$ given by (\ref{T}) contains the dependence on
$z$.

Our purpose is to reduce this expression to a form suitable to use
Lemma (\ref{LEMonedim}). To do this, we first claim that
\begin{equation}\label{reduceclaim}
J \leq 2^{p+2} \int_{B_X(x_0,(3T +\lambda)R)}\!\!\!\!\!\! dx
\int_{-\frac{1}{2}}^{\frac{1}{2}}dt \int_{-\frac{1}{2}}^{\frac{1}{2}}ds |f(x\cdot\gamma(t)) -
f(x\cdot\gamma(s))|^p
\end{equation}
where we have set $\gamma(t) = \exp\left(4 a T R t X\right)$. We
observe that this claim does not contain anymore integration on $z$,
so it will provide an estimate where the multi-indexes $j_n$ and
$\omega_n$ do not appear. To prove it, we begin with
\begin{eqnarray*}
\big|\!\!\!\!\!\!&&\!\!\!\!\!\!\!\!\!f(\eta\cdot\exp(\tau X))-f(\eta)\big|^p\\
& = & \frac{1}{(2aTR)^2} \int_{-aTR}^{aTR}\!\! dt \int_{-aTR}^{aTR}\!\! d\sigma \ \big|f(\eta\cdot\exp(\tau X))-f(\eta)\big|^p\\
& \leq & \frac{2^{p-1}}{(2aTR)^2} \left[ \int_{-aTR}^{aTR}\!\! dt \int_{-aTR}^{aTR}\!\! d\sigma \
\big|f(\eta\cdot\exp(\tau
X))-f(\eta\cdot\exp((\tau + \sigma) X))\big|^p\right.\\
&& \left. + \ \int_{-aTR}^{aTR}\!\! dt \int_{-aTR}^{aTR}\!\! d\sigma \ \big|f(\eta\cdot\exp((\tau + \sigma)
X))-f(\eta)\big|^p\right]\\
& = & \frac{2^{p-1}}{(2aTR)^2}\big[A(\eta,\tau) + B(\eta,\tau)\big]\
.
\end{eqnarray*}
We now look separately at these two terms: for the first one we have, indicating $\omega = aTR$ and $B_X = B_X(x_0,R)$
\begin{eqnarray*}
J_A(\tau)\!\!\! & = & \!\!\!\int_{B_X(x_0,(T+\lambda)R)}\!\!\!\!\!\! d\eta \ A(\eta,\tau)\\
& = &\!\!\! \int_{-\omega}^{\omega}\!\! dt \int_{-\omega}^{\omega}\!\! d\sigma \int_{(T+\lambda)B_X}\!\!\!\!\!\! d\eta \
\big|f(\eta\cdot\exp(\tau X)) - f(\eta\cdot\exp((\tau + \sigma)X))\big|^p\\
& = & \!\!\! \int_{-\omega}^{\omega}\!\! dt \int_{-\omega}^{\omega}\!\! d\sigma \int_{(T+\lambda)B_X}\!\!\!\!\!\!
d\eta \ \big|f(\eta\cdot\exp((\tau-t)X)\cdot\exp(tX))\\
&& \phantom{\int_{-\omega}^{\omega}\!\! dt \int_{-\omega}^{\omega}\!\! d\sigma} -
f(\eta\cdot\exp((\tau - t) X)\cdot\exp((t + \sigma)
X))\big|^p\\
& \leq & \!\!\! \int_{-\omega}^{\omega}\!\! dt \int_{-\omega}^{\omega}\!\! d\sigma \int_{(3T+\lambda)B_X}\!\!\!\!\!\! dx \
\big|f(x\cdot\exp(tX)) - f(x\cdot\exp((t + \sigma) X))\big|^p
\end{eqnarray*}
where the last transiton holds since, after performing the change of
variables $x = \eta\cdot\exp((\tau-t)X)$, we have
$$
d_X(x,x_0) \leq d_X(x,\eta) + d_X(\eta,x_0) < d_X(\exp((\tau-t)X),0)
+ (T+\lambda)R
$$
and $|\tau - t|\leq 2aTR$ \ so, again by (\ref{Control}),
$d_X(\exp((\tau-t)X),0)\leq 2TR$.

Moreover, if we set $s = t + \sigma$, we end up with
\begin{equation}\label{JA}
J_A \leq \int_{-2\omega}^{2\omega}\!\! dt \int_{-2\omega}^{2\omega}\!\! ds
\int_{(3T +\lambda)B_X}\!\!\!\!\!\! dx \ \big|f(x\cdot\exp(tX)) - f(x\cdot\exp(sX))\big|^p
\end{equation}
where, as desired, we have no more dependence on $\tau$, hence on
$z$.

For the second term we proceed in a similar way:
\begin{eqnarray*}
J_B(\tau)\!\!\! & = & \!\!\!\int_{B_X(x_0,(T +\lambda)R)}\!\!\!\!\!\! d\eta \ B(\eta,\tau)\\
& = &\!\!\! \int_{-\omega}^{\omega}\!\! dt \int_{-\omega}^{\omega}\!\! d\sigma \int_{(T+\lambda)B_X}\!\!\!\!\!\! d\eta \
\big|f(\eta\cdot\exp((\sigma + \tau) X) - f(\eta)\big|^p\\
& \leq & \!\!\! \int_{-\omega}^{\omega}\!\! dt \int_{-\omega}^{\omega}\!\! d\sigma \int_{(3T+\lambda)B_X}\!\!\!\!\!\! dx
\big|f(x\cdot\exp((t + \tau) X) - f(x\cdot\exp((t-\sigma) X)\big|^p
\end{eqnarray*}
where in the last transition we have made the change $x =
\eta\cdot\exp((\sigma -t)X)$ so that, with a translation in $\sigma$
and $t$, we get
\begin{equation}\label{JB}
J_B \leq \int_{-2\omega}^{2\omega}\!\! dt \int_{-2\omega}^{2\omega}\!\! ds
\int_{(3T +\lambda)B_X}\!\!\!\!\!\! dx \ \big|f(x\cdot\exp(tX)) - f(x\cdot\exp(sX))\big|^p\ .
\end{equation}
Since
$$
J \leq \frac{2^{p-1}}{(2aTR)^2}\mean_{B_X(0,2\lambda
R)}\!\!\!\!\!\!dz \ (J_A + J_B)
$$
combining (\ref{JA}) and (\ref{JB}) claim (\ref{reduceclaim}) is proved.

We can now make use of the one dimensional estimate
(\ref{onedimpoinc01}), in the form given by Remark \ref{REMsigma}:
indeed claim (\ref{reduceclaim}) allows to reproduce the proof of
Lemma \ref{LEMonedim}, starting from (\ref{righthere}). Hence
$$
J \leq \frac{2^{p+3}}{\int_0^{\sigma} \varphi(\tau)d\tau} \ \sigma
\int_{B} dx \int_{-\frac{1}{2}}^{\frac{1}{2}}dt \int_{-\frac{1}{2}}^{\frac{1}{2}}ds
\frac{|f(x\cdot\gamma(t)) -
f(x\cdot\gamma(s))|^p}{|t-s|^p}\varphi(\sigma|t-s|)
$$
with $B = B_X(x_0,(3T + \lambda)R)$, for any nonincreasing
$\varphi:\R^+\rightarrow\R^+$ in $L^1_{loc}$.

We remark that it is possible to apply Lemma \ref{LEMonedim} since
the request of $f$ to be in $L^p_{loc}(\G)$ guarantees that the
function
$$
t \mapsto f(x\cdot \gamma(t))
$$
is in $L^p_{loc}(\R^+)$ for almost every $x$. This is a consequence
of Fubini theorem: indeed let $B \subset \G$ be a ball and $I
\subset \R^+$ be an interval of the real line, then
\begin{eqnarray*}
\int_B dx \int_I dt |f(x\cdot \gamma(t))|^p & = & \int_I dt \int_B
dx |f(x\cdot \gamma(t))|^p = \int_I dt \int_{B\cdot \gamma(t)} \!\!\!\!\!dy \,
|f(y)|^p\\
& \leq & \int_I dt \int_{B'} dy |f(y)|^p \leq |I| \int_{B'} dy
|f(y)|^p < \infty
\end{eqnarray*}
where $B'$ is a ball containing $B\cdot\gamma(t)$ for all $t \in I$.

At this point we still leave the scale $\sigma$ undetermined, since it will become clear later the exact
amount we will need. Inequality (\ref{POINT}) reduces then to
\begin{equation}\label{interm}
I \leq C \frac{\sigma}{\int_0^{\sigma} \varphi(\tau)d\tau} 
\int_{-\frac{1}{2}}^{\frac{1}{2}}dt \int_{-\frac{1}{2}}^{\frac{1}{2}}ds \
\frac{\varphi(\sigma|t-s|)}{|t-s|^p} \sum_{n=1}^{M} J_n(t,s)
\end{equation}
with
$$
J_n(t,s) = \int_{B} dx |f(x\cdot\gamma_n(t)) -
f(x\cdot\gamma_n(s))|^p
$$
where we have denoted
$$
\gamma_n(t) = \exp\left(4 a T R \, t \, X_{i_{n}}\right)\ .
$$

Let us now define, for arbitrary $\epsilon$, the sets
$$
\Omega_n = B_X\left(\gamma_n\left(\frac{t+s}{2}\right), \epsilon
|t-s|R\right)\ .
$$
These $\Omega_n$ are balls centered in
between $\gamma_n(t)$ and $\gamma_n(s)$, with radius comparable to
their distance, and such that $\Omega_n = \Omega_n(t,s) =
\Omega_n(s,t)$. We then see that
\begin{eqnarray*}
&&\!\!\!\!\!\!\!\!\!\! |f(x\cdot\gamma_n(t)) -
f(x\cdot\gamma_n(s))|^p =
\mean_{\Omega_n} d\xi |f(x\cdot\gamma_n(t)) - f(x\cdot\gamma_n(s))|^p\\
&&\!\!\!\!\!\!\!\!\!\! \leq 2^{p-1}\left(\mean_{\Omega_n} d\xi
|f(x\cdot\gamma_n(t)) - f(x\cdot\xi)|^p + \mean_{\Omega_n} d\xi
|f(x\cdot\gamma_n(s)) - f(x\cdot\xi)|^p\right)\ .
\end{eqnarray*}
Due to the symmetry by exchange of $t$ with $s$, we then obtain
$$
J_n(t,s) \leq 2^p \int_{B} dx \mean_{\Omega_n} d\xi
|f(x\cdot\gamma_n(t)) - f(x\cdot\xi)|^p\ .
$$
We can now perform two changes of variables. The first one is $y =
x\cdot\gamma_n(t)$, so that
\begin{eqnarray*}
J_n(t,s) & \leq & 2^p \mean_{\Omega_n} d\xi \int_{B\cdot\gamma_n(t)}
dy |f(y) - f(y\cdot\gamma_n(t)^{-1}\cdot\xi)|^p\\
& \leq & 2^p \mean_{\Omega_n} d\xi \int_{B_X(x_0,(5T+\lambda)R)} dy
|f(y) - f(y\cdot\gamma_n(t)^{-1}\cdot\xi)|^p
\end{eqnarray*}
where the last transition is due to
\begin{displaymath}
d_X(y,x_0) \leq d_X(y,x) + d_X(x,x_0) < d_X(\gamma_n(t),0) + (3T +
\lambda)R
\end{displaymath}
and, using (\ref{Control}), since $|4aTR\,t|<2aTR$ then
$d_X(\gamma_n(t),0) < 2T R$.

The second change is $h = \gamma_n(t)^{-1}\cdot\xi$, so
that
\begin{eqnarray*}
J_n(t,s) & \leq & 2^p \int_{B_X(x_0,(5T + \lambda)R)} dy
\mean_{\gamma_n(t)^{-1}\cdot\Omega_n} dh
|f(y\cdot h) - f(y)|^p\\
& \leq & 2^p\int_{B_X(x_0,(5T
+\lambda)R)}\!\!\!\!\!\! dy \int_{B_X(0,(\epsilon + 2a
T)|t-s|R)}\!\!\!\!\!\! dh \frac{|f(y\cdot h) - f(y)|^p}{c_{B_X}(\epsilon|t-s|R)^Q}
\end{eqnarray*}
where the last transition holds because
\begin{eqnarray*}
d_X(h,0) & = & d_X(\gamma_n(t)^{-1}\cdot \xi,0) = d_X(\xi,\gamma_n(t))\\
& \leq & d_X\left(\xi,\gamma_n\left(\frac{t+s}{2}\right)\right) +
d_X\left(\gamma_n\left(\frac{t+s}{2}\right),\gamma_n(t)\right)\\
& < & \epsilon |t-s| R +
d_X\left(\gamma_n\left(\frac{|t-s|}{2}\right),0\right)
\end{eqnarray*}
and
\begin{eqnarray*}
d_X\left(\gamma_n\left(\frac{|t-s|}{2}\right),0\right) & \leq &
\int_0^{\frac{|t-s|}{2}} |\dot{\gamma_n}(\tau)|_X d\tau\\
& = & 4 a T R\int_0^{\frac{|t-s|}{2}} |X_{i_n}(\gamma_n(\tau))|_X d\tau = 2 a T |t-s| R\ .
\end{eqnarray*}
We have then obtained for $J_n$ a uniform estimate in $n$,
so (\ref{interm}) becomes
\begin{eqnarray*}
I & \leq & C\frac{\sigma}{R^Q\int_0^{\sigma}
\varphi(\tau)d\tau} \ \int_{-\frac{1}{2}}^{\frac{1}{2}}dt \int_{-\frac{1}{2}}^{\frac{1}{2}}ds \
\frac{\varphi(\sigma|t-s|)}{|t-s|^{p+Q}}\\
&& \int_{B_X(x_0,(5T+\lambda)R)}\!\!\! dy
\int_{B_X(0,(\epsilon + 2a T)|t-s|R)}\!\!\! dh |f(y\cdot h) -
f(y)|^p\\
& \leq & C \frac{\sigma}{R^Q\int_0^{\sigma} \varphi(\tau)d\tau} \
\int_{B_X(x_0,(5T+\lambda)R)}\!\!\! dy \int_0^1\!d\tau \
\frac{\varphi(\sigma\tau)}{\tau^{p+Q}}\\
&& \phantom{\frac{C \ \sigma}{R^Q\int_0^{\sigma} \varphi(\tau)d\tau} \ }\int_{B_X(0,(\epsilon + 2a
T)\tau R)}\!\!\! dh |f(y\cdot h) - f(y)|^p\ .
\end{eqnarray*}

We apply now Fubini theorem to the last two integrations:
\begin{displaymath}
\int_0^1 d\tau \int_{B_X(0,(\epsilon + 2a T)\tau R)} dh =
\int_{B_X(0,(\epsilon + 2a T) R)} dh \int_{\frac{d_X(h,0)}{(\epsilon
+ 2a T)R} }^1 d\tau
\end{displaymath}
so that in the end we get, setting $\nu = \max\{5T+\lambda,\epsilon + 2a T\}$
\begin{eqnarray}\label{final}
I & \leq & \frac{C}{\int_0^{\sigma} \varphi(\tau)d\tau} \
\int_{B_X(x_0,\nu R)}\!\!\!\! dy \int_{B_X(0,\nu R)}\!\!\!\!\! dh |f(y\cdot h) - f(y)|^p
\Phi_{R,\sigma}^{(\nu)}(\n{h})\nonumber
\end{eqnarray}
where
\begin{eqnarray*}
\Phi_{R,\sigma}^{(\nu)}(\delta) & = &
\frac{\sigma}{R^Q}\int_{\frac{\delta}{\nu R} }^1 d\tau
\frac{\varphi(\sigma\tau)}{\tau^{p+Q}} = \frac{\sigma}{R^Q}
\left(\frac{\nu R}{\delta}\right)^{p+Q-1} \int_1^{\frac{\nu
R}{\delta}} dr \frac{\varphi(\frac{\sigma}{\nu R} \delta \,
r)}{r^{p+Q}}\\
& \leq & \frac{\sigma (\nu R)^{p+Q-1}}{R^Q} \frac{1}{\delta^p}
\frac{\varphi(\frac{\sigma}{\nu R} \delta)}{\delta^{Q-1}}
\int_1^{\infty} dr \frac{1}{r^{p+Q}}\\
& = & \frac{1}{p+Q-1} \frac{\sigma (\nu R)^{p+Q-1}}{R^Q}
\frac{1}{\delta^p} \frac{\varphi(\frac{\sigma}{\nu R}
\delta)}{\delta^{Q-1}}\ .
\end{eqnarray*}
If we now set the scale $\sigma$ to $\sigma=\lambda\nu R$, we obtain
$$
\Phi_{R,\lambda\nu R}^{(\nu)}(\delta) \leq
\frac{\lambda\nu^{p+Q}}{p+Q-1} \ R^p \ \frac{1}{\delta^p} \
\frac{\varphi(\lambda\delta)}{\delta^{Q-1}}
$$
which provides the desired result, indeed (\ref{final}) becomes
\begin{eqnarray*}
I & \leq & C R^p\int_{B_X(x_0,\nu R)}\!\!\! dy
\int_{B_X(0,\nu R)}\!\!\! dh \frac{|f(y\cdot h)
- f(y)|^p}{d_X(h,0)^{p+Q-1}} \frac{\varphi(\lambda
d_X(h,0))}{\int_0^{(\epsilon + 2aT) \lambda R}\varphi(\tau)d\tau}\\
& \leq & C_{p,Q} R^p\int_{B(x_0,\lambda\nu R)}\!\!\!
dy \int_{B(0,\lambda\nu R)}\!\!\! dh \frac{|f(y\cdot h) - f(y)|^p}{\n{h}^{p+Q-1}}
\frac{\varphi(\n{h})}{\int_0^{\beta R}\varphi(\tau)d\tau}\ .
\end{eqnarray*}
\end{proof}

\begin{cor}[Poincar\'{e}-Ponce inequality]
Let $\G$ be a Carnot group, $\n{.}$ a homogeneous norm with
parameter $\lambda$ as in (\ref{equiv}), and $\{\rho_n\}$ a family
of radial mollifiers as in Definition \ref{Radmol} such that each
$\rho_n$ is nonincreasing. Then for all balls $B$ of radius $R > 0$
and all $C > C_{p,Q}$ there exists an $n_0$ such that
\begin{equation}
\int_{B}dx |f(x)-f_B| \leq C R^p \int_{\mu B} dx \int_{\mu B} dy
\frac{|f(x) - f(y)|^p}{\n{y^{-1}\cdot x}^p} \moll{y^{-1}\cdot x}
\end{equation}
for all $n > n_0$ and all $f$ in $L^p_{loc}(\G)$, $1\leq p <
\infty$, where $n_0$ is determined by
$$
\int_0^{\beta R} \rho_n(\tau) \tau^{Q-1} d\tau > \frac{C_{p,Q}}{C}
\fa n>n_0
$$
and constants $\mu, \beta$ are as in Theorem \ref{PPteo}.
\end{cor}
\begin{proof}
The proof follows by applying inequality (\ref{PP}) to a family of
one-di\-men\-sio\-nal mollifiers $\{\rho^{(1)}_n\}$ as given by
(\ref{Onedimmoll}) in place of the function $\varphi$, with the
additional requirement that each $\rho^{(1)}_n$ is nonincreasing,
and using relation (\ref{1toQ}). This is the same as we did for
Corollary \ref{cor1Poinc}.
\end{proof}

We can see that the dependence of the threshold $n_0$ on
the dimension of the domain of evaluation is intrinsic in this
approach of localization of finite differences, as already pointed
out after Corollary \ref{cor1Poinc}. In particular, to get a
Poincar\'e-like estimate, the localizing scale of the mollifiers
must tighten as the balls on which the oscillations of the function
are evaluated, at the left hand side, tighten.

We note in addition that, if we choose a homogeneous norm which is
invariant under horizontal rotations, by the result of the previous
section we can recover the classical Poincar\'e inequality by taking
the limit $n \to \infty$.

We conclude the section extending to Carnot groups a
result of \cite{Po}, \cite{BBM} which involves fractional Sobolev
norms. Namely, we can obtain from (\ref{PP}) a Poincar\'e inequality
for norms of Gagliardo type. For relationships among different fractional Sobolev norms in the Euclidean setting see for instance \cite{A}, \cite{Mz} and \cite{LM}.
\begin{cor}\label{corfrac}
Let $\G$ be a Carnot group and $\n{.}$ a homogeneous norm with
parameter $\lambda$ as in (\ref{equiv}). Then
\begin{equation}\label{Fractional}
\int_{B} dx |f(x) - f_B|^p \leq C_{p,Q,s} (1-s)p \ R^{sp}\int_{\mu
B} dx \int_{\mu B}dy \ \frac{|f(y) - f(x)|^p}{\n{y^{-1}\cdot x}^{Q+sp}}
\end{equation}
for all balls $B$ of radius $R$, all $f$ in $L^p_{loc}(\G), \ p\geq
1$, and all $s \in \displaystyle{\left[1-\frac{1}{p} \ ,1\right)}$,
where constants $\mu$ and $\beta$ are as in Theorem \ref{PPteo} and
$$
C_{p,Q,s} = C_{p,Q} \ \beta^{(1-s)p}\ .
$$
\end{cor}
\begin{proof}
Let us choose a function
$$
\varphi(\tau) = \frac{\tau^{Q-1}}{\tau^{Q-(1-s)p}} =
\frac{1}{\tau^{1-(1-s)p}}\ .
$$
If $s \geq 1 - 1/p$ then $\varphi$ is nonincreasing, and if $s<1$
$\varphi$ is in $L^1_{loc}$, so we can apply inequality (\ref{PP})
and the proof follows by direct computation.
\end{proof}

Inequality (\ref{Fractional}) possesses a
self improving property which can be exploited by making use of the
very general results contained in \cite{FPW}. Moreover, by making
use of a technique described in \cite{FGW}, it is possible to
strengthen the inequality with a reduction of the domain of
integration on the right, which we perform at first.\\
We recall here the content of these two results, restricted to the
case of interest for our applications. First, we state Theorem 5.2
of \cite{FGW}, which consists of an application of the Boman chain
technique.
\begin{defi}[Boman chain condition]
We say that a domain $\Omega$ satisfies a Boman chain condition
$\mathcal{F}(\tau,M)$ for some $\tau \geq 1$ and $M \geq 1$ if there
exists a covering $W$ of $\Omega$ consisting of balls $B$ such that
\begin{enumerate}
\item $\tau$-dilated balls have $M$-finite overlapping:
$\sum_{B\in W} \chi_{\tau B}(x) \leq M \chi_{\Omega}(x)$;
\item there is a central ball $B^* \in W$ such that for all balls $B \in
W$ there exists a finite chain of balls $\{B_j\}_{j=1}^{l(B)}$, with
$B_1=B^*$ and $B_{l(B)}=B$ such that
\begin{itemize}
\item[(i)] $B \subset MB_j$;
\item[(ii)] there exists a family of balls $\{R_j\}_{j=2}^{l(B)}$
such that $R_j \subset B_j\cap B_{j-1}$ and $ B_j\cup B_{j-1}
\subset MR_j$.
\end{itemize}
\end{enumerate}
\end{defi}

The notions of Boman domain and John domain are intimately connected,
and in particular Carnot-Ca\-ra\-th\'eo\-do\-ry balls are Boman domains
(see \cite{BKL}).
\begin{teo}[Franchi, Guti\'errez, Wheeden]\label{thFGW}
Let $(X,d,\mu)$ be a quasimetric space with $\mu$ doubling, $\Omega
\subset X$ a domain satisfying a Boman chain condition
$\mathcal{F}(\tau,M)$, and fix $1 \leq p \leq q < \infty$. If $f$
and $g$ are measurable functions on $\Omega$ such that for any ball
$B$ with $\tau B \subset \Omega$ it holds
$$
\|f - f_B\|_{L^q(B)} \leq A \|g\|_{L^p(\tau B)}
$$
with $A$ independent on $B$, then there exists a constant $c_0 =
c_0(\tau,M,q)$ such that
$$
\|f - f_B\|_{L^q(\Omega)} \leq c A \|g\|_{L^p(\Omega)} \ .
$$
\end{teo}

If we apply this result to Corollary (\ref{corfrac}) we
get the following
\begin{cor}\label{corFracBB}
Let $\G$ be a Carnot group, then
\begin{equation}\label{FractionalBB}
\int_{B_X} dx |f(x) - f_{B_X}|^p \leq C (1-s)p \ R^{sp}\int_{B_X}
dx \int_{B_X} dy \ \frac{|f(y) - f(x)|^p}{d_X(y^{-1}\cdot x)^{Q+sp}}
\end{equation}
for all CC balls $B_X$ of radius $R$, all $f$ in $L^p_{loc}(\G), \
1\leq p < \infty$, and all $s \in \displaystyle{\left[1-\frac{1}{p}
\ ,1\right)}$, where constants $\mu$ and $\beta$ are as in Theorem
\ref{PPteo} and
$$
C = c_0^2 C_{p,Q} \ \beta^{(1-s)p} \mu^{-sp}
$$
with $c_0$ coming from Theorem \ref{thFGW}. All constants are
intended with $\lambda = 1$.
\end{cor}
\begin{proof}
If we take Corollary (\ref{corfrac}) with $d_X(x,0)$ as homogeneous
norm, we clearly have $\lambda = 1$. Then, for any fixed $R>0$ and
$x_0 \in \G$ we have that $\Omega = B_X(x_0,R)$ is a Boman domain
for all $\tau$, and in particular for $\tau = \mu$. If we then take
balls $B_j = B_X(x_j,r_j)$ such that $\mu B_j \subset \Omega$, we
have $r_j < \displaystyle{\frac{R}{\mu}}$, since the diameter of CC
balls of radius $r$ is exactly $2r$. Then inequality
(\ref{Fractional}) for the family $\{B_j\}$ reads
$$
\int_{B_j} dx |f(x) - f_{B_X}|^p \leq A \int_{\mu B_j} dx \int_{\mu B_j} dy \
\frac{|f(y) - f(x)|^p}{d_X(y^{-1}\cdot x)^{Q+sp}}
$$
with $A = C_{p,Q,s} (1-s)p \
\displaystyle{\left(\frac{R}{\mu}\right)^{sp}}$, which does not
depend on $B_j$.\\
We now apply Theorem \ref{thFGW} twice: the first time to $g_1(x)
= \|u(x,.)\|_{L^p(\mu B_j)}$ and the second time to $g_2(x) =
\|u(.,y)\|_{L^p(B_j)}$, where $u(x,y)$ stands for $u(x,y) =
\displaystyle{\frac{|f(y) - f(x)|^p}{d_X(y^{-1}\cdot x)^{Q+sp}}}$.
\end{proof}

We recall now Theorem 2.3, Corollary 2.4 and Remark 2.6 of
\cite{FPW}.
\begin{teo}[Franchi, P\'{e}rez, Wheeden]\label{FPWteo}
Let $(X,d,\mu)$ be a metric measure space such that the distance $d$
has the segment property, \emph{i.e.} for all $x,y
\in X$ there exists a continuous curve $\gamma:[0,L]\rightarrow X$ joining $x$ to $y$
such that $d(\gamma(t),\gamma(s)) = |t-s|$ for all $0\leq t,s \leq L$, and let $\mu$ be
a doubling measure, \emph{i.e.} there exists a constant $C_d$ such
that
$$
0<\mu(B(x,2r))\leq C_d \mu(B(x,r))<\infty
$$ for all $x \in X$ and all $r > 0$, where $B$ stands for the open
ball with respect to $d$. We will denote by $\mathcal{B}$ the class
of such balls.\\
Suppose that
\begin{itemize}
\item[i)] $a:\mathcal{B}\rightarrow \R^+$ is a functional on balls
for which there exists a constant $c < \infty$ and a constant $1
\leq r < \infty$ such that it holds
\begin{equation}\label{DrCondition}
\sum_i a(B_i)^r \mu(B_i) \leq c^r a(B)^r \mu(B)
\end{equation}
for any ball $B$ and any family $\{B_i\}$ of subballs of $B$ with
bounded overlaps, that is there exist an integer $M \geq 0$ such
that for each $B_i$ there are at most $M$ other balls of the same
family intersecting it;
\item[ii)] $f:B_0\rightarrow \R$ is a function, defined on a ball
$B_0$, satisfying
$$
\mean_{B} |f-f_B|d\mu \leq a(B)
$$
for any ball $B \subset B_0$.
\end{itemize}
Then there exists a constant $c'$ independent of $f$ and $B_0$ such
that
\begin{equation}\label{FPW}
\left(\mean_{B_0} |f-f_{B_0}|^p d\mu\right)^{1/p} \leq c' \|a\|
a(B_0)
\end{equation}
for any $1<p<r$, where we have denoted by $\|a\|$ the smallest
constant $c$ such that (\ref{DrCondition}) holds.
\end{teo}

\begin{cor}\label{CorFracImproved}
Let $\G$ be a Carnot group, then
\begin{displaymath}\label{FracImproved}
\bigg(\mean_{B_X} dx |f(x) - f_{B_X}|^p\bigg)^{1/p} \leq C (1-s)
R^{Q+s} \mean_{B_X} dx \mean_{B_X} dy \frac{|f(x) -
f(y)|}{\n{y^{-1}\cdot x}^{Q+s}}
\end{displaymath}
for all CC balls $B_X$ of radius $R$, all $s \in (0 ,1)$, all $f \in
L^1_{loc}(\G)$ and all $1\leq p < p^*$, where the critical exponent
is given by $p^* = \displaystyle{\frac{Q}{Q-s}}$\ .
\end{cor}
\begin{proof}
Inequality (\ref{FracImproved}) is a direct consequence of Theorem
\ref{FPWteo} applied to Corollary \ref{corFracBB}. If we take $p=1$
in (\ref{FractionalBB}) we get
\begin{displaymath}
\mean_{B_X(x_0,R)}\!\!\!\!dx |f(x) - f_{B_X}| \leq C (1-s) R^{s-Q}
\int_{B_X(x_0, R)}\!\!\!\!dx \int_{B_X(x_0, R)}\!\!\!\!dy \frac{|f(x) - f(y)|}{d_X(y^{-1}\cdot
x)^{Q+s}}
\end{displaymath}
The requirement of having balls satisfying the segment property is
attained for geodesic balls, and the doubling condition is satisfied
by homogeneity, so we need only to verify condition
(\ref{DrCondition}) for the functional
$$
a(B_X(x_0,R)) = R^{s-Q} \int_{B_X(x_0,R)}\!\!dx
\int_{B_X(x_0,R)}\!\!dy g(x,y)
$$
where
$$
g(x,y) = \frac{|f(x) - f(y)|}{d_X(y^{-1}\cdot x)^{Q+s}}
$$
in order to do that, let us take a family $\{B_i\}$ of $M$-finite overlapping balls $B_i
\subset B_X=B_X(x_0,R)$ of radius $R_i$. Then
\begin{eqnarray*}
\sum_i a(B_i)^rR_i^Q & = & \sum_i \left(\int_{B_i}\!\!dx
\int_{B_i}\!\!dy \ g(x,y)\right)^r R_i^{(s-Q)r+Q}\\
& \leq & R^{(s-Q)r+Q} \sum_i \left(\int_{B_X}\!\!dx
\int_{B_X}\!\!dy \ g(x,y)\chi_{{}_{B_i}}(x)\chi_{{}_{B_i}}(y)\right)^r\\
& \leq & R^{(s-Q)r+Q} \left(\sum_i \int_{B_X}\!\!dx
\int_{B_X}\!\!dy \ g(x,y)\chi_{{}_{B_i}}(x)\chi_{{}_{B_i}}(y)\right)^r\\
& \leq & M^{2r} \left(a(B_X)\right)^r R^Q
\end{eqnarray*}
indeed, the first inequality holds if $(s-Q)r+Q > 0$, and the last
one is due to the finite overlapping of the family of balls,
provided $r \geq 1$.
\end{proof}


\end{document}